\documentclass
[
11pt]
{amsart}%

\usepackage{hyperref}
\usepackage{cite}
\usepackage{amssymb}
\usepackage{amsmath}
\usepackage{amsthm}
\usepackage{amsfonts}
\usepackage{graphicx}
\usepackage{bbm}
\usepackage{color}
\usepackage[toc,page]{appendix}
\usepackage{subfigure}

\usepackage{tikz, tikz-cd}
\usetikzlibrary{matrix,arrows,decorations.pathmorphing}

\newcommand{\beq}{\begin{equation}}
\newcommand{\eeq}{\end{equation}}
\newcommand{\bea}{\begin{eqnarray}}
\newcommand{\para}{\mathbin{\!/\mkern-5mu/\!}}
\newcommand{\eea}{\end{eqnarray}}

\newtheoremstyle{mystyle}
  {}
  {}
  {\itshape}
  {}
  {\bfseries}
  {.}
  { }
  {\thmname{#1}\thmnumber{ #2}\thmnote{ (#3)}}

\theoremstyle{mystyle}

\newtheoremstyle{mystyledefinition}
  {}
  {}
  {}
  {}
  {\bfseries}
  {.}
  { }
  {\thmname{#1}\thmnumber{ #2}\thmnote{ (#3)}}

\theoremstyle{mystyledefinition}

\newtheorem{thm}{Theorem}[section]

\newtheorem{corollary}[thm]{Corollary}
\newtheorem{lemma}[thm]{Lemma}

\newtheorem{assumption}{Assumption}
\newtheorem{lem}[thm]{Lemma}
\newtheorem{prop}[thm]{Proposition}
\newtheorem{definition}[thm]{Definition}
\newtheorem{proposition}[thm]{Proposition}

\newtheorem{theorem}{Theorem}
\theoremstyle{mystyledefinition}

\newtheorem{example}{Example}[section]

\newtheorem{rem}[thm]{Remark}
\newtheorem{remark}[thm]{Remark}
\newtheorem{notation}[thm]{Notation}
\numberwithin{equation}{section}
\newtheorem*{acknowledgement}{Acknowledgement}

\newcommand{\ee}{\ell}

\newcommand{\bM}{\mathbb M}

\newcommand{\Ho}{\mathcal H}

\newcommand{\M}{\mathbb M}
\newcommand{\B}{\mathbb B}

\newcommand{\R}{\mathbb R}

\newcommand{\ve}{\varepsilon}

\begin{document}

\title[Integration by parts]{Integration by parts and quasi-invariance for the horizontal Wiener measure on foliated compact manifolds}

\author[Baudoin]{Fabrice Baudoin{$^{\dag}$}}
\thanks{\footnotemark {$\dag$} Research was supported in part by NSF Grant DMS-1660031}
\address{$^{\dag} $
Department of Mathematics
\\
University of Connecticut
\\
Storrs, CT 06269,  U.S.A.}

\author[Feng]{Qi Feng{$^{\ddag}$}}
\address{$^{\ddag}$
Department of Mathematics
\\
University of Southern California
\\
Los Angeles, CA 90089-2532,  U.S.A.}

\author[Gordina]{Maria Gordina{$^{\dag\dag}$}}
\thanks{\footnotemark {$\dag\dag$} Research was supported in part by NSF Grant DMS-1405169, DMS-1712427 and the Simons Fellowship.}
\address{$^{\dag\dag}$
Department of Mathematics
\\
University of Connecticut
\\
Storrs, CT 06269,  U.S.A.}
\email{maria.gordina@uconn.edu}

\keywords{Wiener measure; Cameron-Martin theorem; Quasi-invariance; sub-Riemannian geometry, Riemannian foliation}

\subjclass{Primary 58G32; Secondary 58J65, 53C12, 53C17}


\begin{abstract}
We prove several sub-Riemannian versions of  Driver's integration by parts formula which first appeared in \cite{Driver1992b}. Namely, our results are for the horizontal Wiener measure on a totally geodesic Riemannian foliation equipped with a sub-Riemannian structure. It is also shown that the horizontal Wiener measure is quasi-invariant under the action of flows generated by suitable tangent processes.

\end{abstract}

\maketitle

\tableofcontents


\section{Introduction}

\subsection{Background}

In this paper we study quasi-invariance properties and related integration by parts formulas for the horizontal Wiener measure on a foliated Riemannian manifold equipped with a sub-Riemannian structure. These are most closely related to the well-known results by B.~Driver  \cite{Driver1992b} who established such properties for the Wiener measure on a path space over a compact Riemannian manifold. Quasi-invariance in such settings can be viewed as a curved version of the classical Cameron-Martin  theorem for the Euclidean space. While the techniques developed for path spaces over Riemannian manifolds are not easily adaptable to the sub-Riemannian case we consider, we take advantage of the recent advances in this field. The geometric and stochastic analysis of sub-Riemannian structures on foliated manifolds has attracted a lot of attention in the past few years (see for instance \cite{BaudoinKimWang2016, Baudoin2017b, Elworthy2014, GordinaLaetsch2016a, Grong, GrongThalmaier2016a, GrongThalmaier2016b, ThalmaierCIRMLectures2016}).

In particular, we make use of the tools such as Weitzenb\"{o}ck formulas for the sub-Laplacian  extending results by J.-M.~Bismut, B.~Driver et al to foliated Riemannian manifolds. More precisely, the first progress in developing geometric techniques in the sub-Riemannian setting has been made in \cite{BaudoinGarofalo2017}, where  a version of Bochner's  formula for the sub-Laplacian was established and generalized curvature-dimension conditions have been studied. This Bochner-Weitzenb\"ock  formula was then used in \cite{Baudoin2017b} to develop a sub-Riemannian stochastic calculus. One of the difficulties in this case is that, a priori, there is no canonical connection on such manifolds such as the Levi-Civita connection in the Riemannian case.  However,  \cite{Baudoin2017b} introduces a one-parameter family of metric connections associated with Bochner's formula  proved in \cite{BaudoinGarofalo2017} and shows that the derivative of the sub-Riemannian heat semigroup can be expressed in terms of a damped stochastic parallel transport.

It should be noted that these connections do not preserve the geometry of the foliation in general. In particular, the corresponding parallel transport does not necessarily transform a horizontal vector into a horizontal vector, that is, these connections in general are not horizontal. As a consequence, establishing an integration by parts formulas for directional derivatives on the path space of the horizontal Brownian motion, similarly to Driver's integration by parts formula in \cite{Driver1992b} for the Riemannian Brownian motion, is not straightforward. As a result, the integration by parts formula we prove in the current paper can not be simply deduced from the derivative formula for the corresponding semigroup by applying the standard techniques of covariant stochastic analysis on manifolds as presented for instance in \cite[Section 4]{ElworthyLeJanLiBook1999}, in particular  \cite[Theorems 4.1.1, 4.1.2]{ElworthyLeJanLiBook1999}. A different approach to proving quasi-invariance in an infinite-dimensional sub-Riemannian setting has been used in \cite{BaudoinGordinaMelcher2013}.

Analysis on path and loop spaces has been developed over several decades, and we will not be able to refer to all the relevant publications, but we mention some which are closer to the subject and techniques of this paper. In particular, J.-M. Bismut's book \cite{BismutBook1984} contains an integration by parts formula on a path space over a compact Riemannian manifold. His methods were based on the Malliavin calculus and Bismut's motivation was to deal with a hypoelliptic setting as described in {\cite[Section 5]{BismutBook1984}}. A breakthrough has been achieved by B.~Driver \cite{Driver1992b}, who established quasi-invariance properties of the Wiener measure over a compact Riemannian manifold, and as a consequence an integration by parts formula. This work has been simplified and extended by E.~Hsu \cite{Hsu1995b}, and also approached by O. B. Enchev and D. W. Stroock in \cite{EnchevStroock1995a}. A review of these techniques can be found in \cite{Hsu1999a}. In \cite{Hsu2002a, HsuOuyang2009} the noncompact case has been studied. Let us observe here that B.~Driver in \cite{Driver1992b} and later E.~Hsu in \cite{Hsu1995b} have considered connections on a Riemannian manifold which are metric-compatible, but not necessarily torsion-free. This is very relevant in our setting of a foliated Riemannian manifold equipped with a sub-Riemannian structure, because on sub-Riemannian manifolds the natural connections are not torsion-free. A different approach to analysis on Riemannian path space can be found in \cite{CruzeiroMalliavin1996}, where tangent processes, Markovian connections, structure equations and other elements of what the authors call the renormalized differential geometry on the path space have been introduced.

\subsection{Main results and organization of the paper}\label{s.MainResults}


We now explain in more detail our main results without the technical details, and describe how the paper is organized.
Section \ref{p.1} studies quasi-invariance properties for the horizontal Wiener measure of a Riemannian foliation, and in Section \ref{p.2} we prove  integration by parts formulas. Although quasi-invariance properties and integration by parts formulas are intimately related and actually equivalent in many settings (see \cite{Bell2006a, Driver1992b, Driver1997a}), we use very different techniques and approaches in these two sections. To prove quasi-invariance, we develop a stochastic calculus of variations for the horizontal Brownian  motion on a foliation in the spirit of \cite{CruzeiroMalliavin1996, Driver1992b, Hsu1995b}, whereas to prove integration by parts formulas, we shall make use of Markovian techniques and martingale methods as presented for instance in \cite[Section 4]{ElworthyLeJanLiBook1999}.

Let $\left( \M, g, \mathcal{F} \right)$ be a smooth connected compact Riemannian manifold of dimension $n+m$ equipped with a totally geodesic and bundle-like foliation $\mathcal{F}$ by $m$-dimensional leaves as described in Section \ref{s.Notation}. On such manifolds, one can define a horizontal Laplacian $L$ according to \cite[Section 2.2, Section 2.3]{BaudoinEMS2014}), allowing to define a horizontal Brownian motion as the diffusion on $\M$ with generator $\frac{1}{2} L$, as we describe in Section \ref{Dirichlet}. The distribution of the horizontal Brownian motion is called the \emph{horizontal Wiener measure}.

Recall that for a Riemannian manifold $(\M, g)$, for a given metric connection one can construct a development map $B \longmapsto W$, where $B$ is a Brownian motion in $\mathbb R^{n+m}$ and $W$ the Brownian motion on the manifold $(\M,g)$, see for instance \cite[Theorem 3.4]{Driver1992b}.
We construct development maps in the setting of a totally geodesic Riemannian foliation even though we do not have a Levi-Civita connection in this sub-Riemannian setting.

The foliation structure on $\M$ induces a natural splitting of the tangent bundle into a vertical and horizontal subbundles $\mathcal{V}$ and $\mathcal{H}$ as described in Section \ref{ss.HorizVertSubbundles}. This allows us to define horizontal Brownian motion with respect to this structure. In Section \ref{Eels-Elworthy} we show that there exist metric connections on $\M$ which are compatible with the foliation $\mathcal{F}$ in such a way that the above development map sends $(\beta, 0)$  to a horizontal Brownian motion of the foliation, where $\beta$ is a Brownian motion in $\mathbb R^n$. In particular, the horizontal Brownian motion $W$ on $\M$ constructed in this way is a semimartingale on $\M$ and it becomes possible to develop a horizontal stochastic calculus of variations. In this paper, the map $\beta \longmapsto W$ is referred  to as the \emph{horizontal stochastic development map}. The main result of Section \ref{s.HorizontCalculus} is Theorem \ref{intro 1} that characterizes the variations of horizontal paths (i.e. paths transverse to the leaves of the foliation).


Before we can formulate our first main result, we need to describe some of the notation used. For details the reader is referred to Section \ref{s.HorizontCalculus}. Let $D$ be a metric connection on $\left( \M, g, \mathcal{F} \right)$ adapted to the foliation structure as described by Assumption \ref{property D}. An example of such a connection is the Bott connection introduced in Section \ref{Bott connection}.

The first observation is that the connection $D$ allows us to define vector fields on the space $W^\infty_{0}\left( \M \right)$ of smooth $\M$-valued paths  on the interval $[0,1]$ as follows.
For $v \in W^\infty_{0}\left( \mathbb{R}^{n+m} \right)$, the space of smooth $\mathbb{R}^{n+m}$-valued paths, we denote by $\mathbf D_v$ the vector field on the space of smooth paths  $[0,1]\to \M$ defined by
\[
\mathbf D_v (\gamma)_s=u_s(\gamma) v_s,
\]
where $u$ is the $D$-lift of $\gamma$ to the orthonormal frame in  the orthonormal frame bundle $\mathcal{O} (\M)$. In addition, we can use the connection $D$  to introduce the corresponding \emph{development map} $\phi: W^\infty_{0}\left( \mathbb{R}^{n+m} \right) \longrightarrow W^\infty_{0}\left( \M \right)$ as defined in Definition \ref{develop}. The inverse map $\phi^{-1}: W^\infty_{0}\left( \M \right) \longrightarrow W^\infty_{0}\left( \mathbb{R}^{n+m} \right)$ is referred to as an \emph{anti-development map}. We also define horizontal development and anti-development maps in Definition \ref{d.HorizSmoothItoMaps}.

In addition the connection $D$ can be used to lift vector fields on $\M$ to vector fields on $\mathcal{O}(\M)$ consistent with the foliation structure as explained in Notation \ref{HVframe}. We denote by $A, V$ the fundamental   vector fields on $\mathcal{O} (\M)$ associated with this $D$-lift. For details we refer the reader to Notation \ref{HVframe}. As motivation for the semimartingale version, we start with a theorem combining the results in the smooth setting.

\begin{theorem}[Theorem \ref{Hsu deterministic}, Theorem \ref{tangent deterministic}]\label{intro 1}
Let $D$ be a metric connection on $\left( \M, g, \mathcal{F} \right)$ adapted to the foliation structure as described by Assumption \ref{property D}. For a smooth path $v$ on $\mathbb{R}^{n+m}$ , we let $\{ \zeta_t^v, t \in \mathbb{R} \}$ be the flow generated by the vector field $\mathbf D_v$ on $W^\infty_{0}\left( \M \right)$. Then for a smooth horizontal path $\gamma$ on $\M$
\[
\left.\frac{d}{dt}\right|_{ t=0} \phi^{-1} (\zeta_t^v \gamma)_s  \in \mathbb{R}^n, \quad  s \in [0, 1]
\]
if and only if the path
\begin{equation}\label{tangent 1}
 v(s)-\int_0^s T_{u_r} (  A \left(d\omega^\mathcal{H}_r\right),  Av(r)), \quad s \in [0,1]  \text{ takes values in } \mathbb{R}^n,
\end{equation}
that is, it is horizontal. Here $\omega^\mathcal{H}$ is the horizontal anti-development of the horizontal path $\gamma$, and $T$ is the torsion of the Bott connection.

Moreover, if \eqref{tangent 1} is satisfied, then
\[
\left.\frac{d}{dt}\right|_{ t=0} \phi^{-1} (\zeta_t^v \gamma)_s  =p_v(\omega^\mathcal{H})_s,
\]
where
\begin{align*}
 p_v(\omega^\mathcal{H})_s & = v(s)-\int_0^s T^D_{u_r} (  A \left( d\omega^\mathcal{H}_r \right),  Av(r)+Vv(r)) - \\
 & \int_0^s \left(  \int_0^r \Omega^D_{u_\tau}( A   \left(d \omega^\mathcal{H}_\tau\right), A v(\tau)+V v (\tau)) \right)   d \omega^\mathcal{H}_r. \notag
\end{align*}
Here  $T^D$ is the torsion form of the connection $D$ and  $\Omega^D$ its curvature  form.
\end{theorem}
If \eqref{tangent 1} is satisfied, we will say that the path $v$ is \emph{tangent} to the horizontal path $\gamma$. We stress that in  \eqref{tangent 1} we use the torsion of the Bott connection, not the torsion of the connection $D$. Thus Theorem \ref{intro 1} shows that the notion of a tangent path is independent of the connection $D$, as long as it satisfies Assumption \ref{property D}.

Given a horizontal path, it is easy to construct tangent paths to this path. Indeed, we show in Lemma \ref{construction tangent} that if $\omega^\mathcal{H}$ is a smooth path in $\mathbb R^n$ then for every smooth path $h$ in $\mathbb{R}^n$
\begin{align}\label{tau intro}
\tau_h (\omega^\mathcal{H})_{s}:=h(s)+\int_0^s T_{u_r} (  A \left(d\omega^\mathcal{H}_r \right),  Ah(r))
\end{align}
is a tangent path to $\phi(\omega^\mathcal{H})$, where $u$ denotes the lift of  $\phi(\omega^\mathcal{H})$.

In Section \ref{s.QI} we use Malliavin's \textit{principe de transfert} ansatz (see  \cite[Part IV Chapter VIII]{MalliavinStochasticAnalysisBook}) to extend the definitions for $p_v$ and $\tau_h$ to semimartingale paths by replacing integration against smooth paths by Stratonovich stochastic integrals with respect to semimartingales. More precisely, we work on the probability space $\left( W_0(\mathbb{R}^n), \mathcal{B}, \mu_\mathcal{H} \right)$, where $\mathcal{B}$ is the Borel $\sigma$-algebra on the path space $W_0(\mathbb{R}^n)$ of continuous paths $[0,1] \to \mathbb{R}^n$ starting at $0$, $\mu_\mathcal{H}$ is the Wiener measure. The measure $\mu_W$ can be also described as the distribution of the horizontal Brownian motion on $\M$.

Given a deterministic Cameron-Martin path $h:[0, 1]\to \mathbb{R}^n$, one can then consider the $\mathbb{R}^{n+m}$-valued semimartingale
\[
\tau_h (\omega^\mathcal{H})_s:=h(s)+\int_0^s T_{u_r} (  A \circ d\omega^\mathcal{H}_r,  Ah(r)),
\]
where $\omega^\mathcal{H}$ is the coordinate process and $\circ d\omega^\mathcal{H}$ denotes the Stratonovich integral.  Note that $\tau_h$ is defined  $\mu_\mathcal{H}$-a.s. One can then think of $\tau_h (\omega^\mathcal{H})$ as a tangent process to the horizontal Brownian motion of the foliation. We will view $\tau_h: W_0(\mathbb{R}^n) \to W_0(\mathbb{R}^{n+m})$ as an adapted process with respect to the natural filtration $\left\{ \mathcal{B}_{s}, 0 \leqslant s \leqslant 1 \right\}$ generated by the horizontal Brownian motion in $\mathbb{R}^{n+m}$.  Notice that $\tau_h$ is really an equivalence class of processes with two processes being equivalent if they are equal $\mu_\mathcal{H}$-a.s. similarly to \cite[p. 425]{Hsu1995b}.
Thus when we say that a map is  defined $\mu_\mathcal{H}$-a.s. we mean that we are actually working with  equivalence classes of maps. It will be an important part of our results that the flows and the compositions we consider  preserve the equivalence classes we are working with, but for simplicity of the presentation, those considerations will remain in the background in our discussions similarly to \cite{Hsu1995b}. This aspect is discussed more thoroughly in \cite{Driver1992b}.

Similarly, given an $\mathbb{R}^{n+m}$-valued semimartingale $v$, one can define the semimartingale
\begin{align*}
 p_v(\omega^\mathcal{H})_s & = v(s)-\int_0^s T^D_{u_r} (  A \circ d\omega^\mathcal{H}_r,  Av(r)+Vv(r)) - \\
 & \int_0^s \left(  \int_0^r \Omega^D_{u_\tau}( A   \circ d \omega^\mathcal{H}_\tau , A v(\tau)+V v (\tau)) \right)   \circ d \omega^\mathcal{H}_r. \notag
\end{align*}

The main results of Section \ref{s.QI} include Theorem \ref{flow euclidean} and Theorem \ref{main} which can be summarized as follows. Here we use the notion of stochastic horizontal development and anti-development $\phi_\mathcal{H}$ and $\phi_\mathcal{H}^{-1}$  as defined in Definition \ref{d.HorizStochItoMaps}.



\begin{theorem} [Theorem \ref{flow euclidean},  Theorem \ref{main}] \label{intro 2}

There exists a  family of semimartingales $\{ \nu_t^h, t \in \mathbb{R}\}$ such that for each fixed $t$ the random variable $\nu_t^h: W_{0}\left( \mathbb{R}^n \right) \longrightarrow W_{0}\left( \mathbb{R}^n \right)$  can be regarded as a $\mu_\mathcal{H}$-a.~s. defined map from the path space to itself.  In particular, $s\rightarrow \nu_t^h(s)$ is a $\mathbb{R}^n$-valued semi-martingale over the probability space $W_0(\mathbb{R}^n)$. In addition $\nu_t^h$ has the group property. Thus we can regard  $\{ \nu_t^h, t \in \mathbb{R}\}$ as the flow on $W_{0}\left( \mathbb{R}^n \right)$ generated by $p_{\tau_h}: W_{0}\left( \mathbb{R}^n \right) \to W_{0}\left( \mathbb{R}^n \right)$ which is defined $\mu_\mathcal{H}$-a.s. 

Moreover, the measure  $\mu_\mathcal{H}$ is quasi-invariant under this flow, that is, the law $\mu_{W}$ of the horizontal Brownian motion on $\M$  is quasi-invariant under the  $\mu_W$-a.s.  defined flow $\zeta^h_t = \phi_\mathcal{H} \circ \nu_t^h \circ \phi_\mathcal{H}^{-1} : W_{x_0}\left( \mathbb{M} \right) \to W_{x_0}\left( \mathbb{M} \right)$, $t \in \mathbb{R}$. Here $\phi_\mathcal{H}$ and $\phi_\mathcal{H}^{-1}$ are horizontal stochastic development and anti-development map correspondingly.
\end{theorem}
It should be noted that  our argument follows  relatively closely the one by B.~Driver in \cite{Driver1992b} and later by E.~Hsu in \cite{Hsu1995b} (see also \cite{Cruzeiro1983b, CruzeiroMalliavin1996, EnchevStroock1995a}) and therefore going from Theorem \ref{intro 1} to Theorem \ref{intro 2} is quite routine. In Section \ref{Example QI} we illustrate our results in the case of Riemannian submersions and explicitly compute the flow $\zeta^h_t$ associated to the Bott connection in some examples.

The goal of the second part of the paper  is to establish several types of integration by parts formulas for the horizontal Brownian motion.
In Section \ref{W:6}, we survey known geometric and stochastic results and introduce the notation and conventions used throughout Section \ref{p.2}. Most of this material is based on \cite{BaudoinKimWang2016} for the geometric part and  \cite{Baudoin2017b} for the stochastic part. The most relevant result that will be used later is the Weitzenb\"ock formula given in Theorem \ref{Weitzenbock}. It  asserts that for every $f \in C^\infty(\M)$, $x \in \M$ and every $\varepsilon >0$
\begin{equation}\label{W1}
dLf (x)=\square_\varepsilon df (x),
\end{equation}
where $\square_\varepsilon$ is a one-parameter family of sub-Laplacians on one-forms indexed by a parameter $\varepsilon >0$.
These sub-Laplacians on one-forms are constructed from a family of metric connections $\nabla^\varepsilon$ introduced in \cite{Baudoin2017b} whose adjoint connections $\widehat{\nabla}^\varepsilon$ in the sense of B.~Driver in \cite{Driver1992b} are also metric. These connections satisfy Assumption \ref{property D}, so that the results of Section \ref{p.1} are applicable. Even though Section \ref{W:6} introduces mostly preliminaries, we present a number of new results there such as Lemma \ref{ricci adjoint}.

In Section \ref{s.IbyP}, we prove integration by parts formulas for the horizontal Wiener measure with the main result being Theorem \ref{IBP2} which includes the following result. Suppose $F$ is a cylinder function,  $v$ is a tangent  process on $T_{x} \M$ as defined in Definition \ref{admissible}, then we have
\begin{align}\label{IBP intro}
  \mathbb{E}_x\left( \mathbf{D}_v F \right)
 =\mathbb{E}_x \left( F \int_0^1 \left\langle v_\mathcal{H}^{\prime}(s)+ \frac{1}{2} \para_{0,s}^{-1} \mathfrak{Ric}_{\mathcal{H}}  \para_{0,s}, dB_s\right\rangle_{\mathcal{H}}  \right),
\end{align}
where $x$ is the starting point of the horizontal Brownian motion, $\mathbf{D}_v F$ is the directional derivative of $F$ in the direction of $v$, $\para_{0,s}$ is the stochastic parallel transport for the Bott connection, and  $\mathfrak{Ric}_{\mathcal{H}}$ is the horizontal Ricci curvature of the Bott connection. The Bott connection as defined in Section \ref{Bott connection} corresponds to the adjoint connection $\widehat{\nabla}^\varepsilon$ as $\varepsilon \to \infty$. In the integration by parts formula \eqref{IBP intro}, the tangent  process $v$ is a $T_{x} \M$--valued process such that its horizontal part $v_\mathcal{H}$ is absolutely continuous and satisfies $\mathbb{E} \left( \int_0^1 \Vert v^{\prime}_\mathcal{H} (s) \Vert_{T_{x} \M}^2 ds \right)< \infty$ and its vertical part is given by
\begin{align}\label{vertg}
v_\mathcal{V} (s)=\int_0^s \para_{0, r}^{-1}  T (\para_{0, r}  \circ dB_r, \para_{0,r}  v_\mathcal{H} (r)),
\end{align}
where $T$ is the torsion tensor of the Bott connection. Observe that  \eqref{IBP intro} looks similar to the integration by parts formulas by J.-M.~Bismut and B.~Driver. This is not too surprising if one thinks about the special case when the foliation comes from a Riemannian submersion with totally geodesic fibers. We consider this case in Section \ref{ss.RiemSubmersions}, and we prove that then that the integration by parts formula in Theorem ~\eqref{IBP2} is actually a horizontal lift of Driver's formula from the base space of the fibers to $\M$. However, in general foliations do not come from submersions (see for instance  \cite{Escobales1982a} for necessary and sufficient conditions) and one therefore needs to develop an intrinsic horizontal stochastic calculus on $\M$ to prove \eqref{IBP intro}. Developing such a calculus is one of the main accomplishments of the current paper.

The proof of Theorem \ref{IBP2} proceeds in several steps. As in \cite{Baudoin2017b}, the Weitzenb\"ock formula \eqref{W1} yields a stochastic representation for the derivative of the semigroup of the horizontal Brownian motion in terms of a damped stochastic parallel transport associated to the connection $\nabla^\varepsilon$ (see Lemma \ref{l.5.4}). By using techniques of \cite{BaudoinFeng2016},  Lemma \ref{l.5.4} implies an integration by parts formula for the damped Malliavin derivative as stated in Theorem \ref{IBP}. The final step is to prove  Theorem \ref{IBP2} from Theorem \ref{IBP}. The main difficulty is that the connection $\nabla^\varepsilon$ is in general not horizontal. However, it turns out that the adjoint connection $\widehat{\nabla}^\varepsilon$  is not only metric but also horizontal. As a consequence, one can use the orthogonal invariance of the horizontal Brownian motion (Lemma \ref{IPP2}) to filter out the redundant noise which is given by the torsion tensor of $\nabla^\varepsilon$ . It is remarkable that the integration by parts formula   for the directional derivatives in Theorem  \eqref{IBP2} is actually independent of the choice of a particular connection and therefore is independent of $\varepsilon$ in the one-parameter family of connections used to define the damped Malliavin derivative. While integration by parts formulas for the damped Malliavin derivative may be used to prove gradient bounds for the heat semigroup (as in \cite{Baudoin2017b}) and log-Sobolev inequalities on the path space (as in \cite{BaudoinFeng2016}), we prove that the integration by parts formula \eqref{IBP intro} comes from the quasi-invariance property of the horizontal Wiener measure proved in Section \ref{s.QI}.

\begin{remark}
In the current paper, we restrict consideration to the case of compact manifolds mainly for the sake of conciseness. It is reasonable to conjecture that as in \cite{HsuOuyang2009}, our results may be extended to complete manifolds.
\end{remark}

\begin{acknowledgement}
The authors thank Bruce Driver for stimulating discussions and an anonymous referee for insightful remarks that helped to improve the presentation of the paper significantly and to clarify  key definitions.
\end{acknowledgement}

\section{Geometric preliminaries: Riemannian foliations}\label{s.Notation}


\subsection{Riemannian foliations}

We start by recalling the notion of a foliation. Let $\mathbb{M}$ be a smooth connected  manifold of dimension $n+m$. Then a foliation of dimension $m$ on $\mathbb{M}$ is usually described as a collection $\mathcal{F}$ of disjoint connected non-empty immersed $m$-dimensional submanifolds of $\mathbb{M}$ (called the \emph{leaves} of the foliation), whose union is $\mathbb{M}$, and such that in a neighborhood of each point in $\mathbb{M}$ there exists a chart for $\mathcal{F}$ as follows.

Before we define such Riemannian foliations, let us introduce some standard notation.

\begin{notation}\label{n.geometry}
Suppose $\left( \M, g\right)$ is a Riemannian manifold. By $T\M$ we denote the \emph{tangent bundle} and by $T^{\ast}\M$ the \emph{cotangent bundle}, and by  $T_{x}\M$ ($T_{x}^{\ast}\M$) the \emph{tangent (cotangent) space} at $x \in \M$. The inner product on $T\M$ induced by the metric $g$ will be denoted by $g\left( \cdot, \cdot \right)$. If $\mathcal{U}$ is a subbundle of the tangent bundle $T\M$,  the restriction of $g$ to $\mathcal{U}$ will be denoted by $g_{\mathcal{U}}\left( \cdot, \cdot \right)$.

As always, for any $x \in \M$ we denote by $g\left( \cdot, \cdot \right)_{x}$ (or $\langle \cdot, \cdot \rangle_x$), $g_{\mathcal{U}}\left( \cdot, \cdot \right)_{x}$ (or $\langle \cdot, \cdot \rangle_{\mathcal{U}_x}$) (or $\langle \cdot, \cdot \rangle_{\mathcal{U}_x}$) the inner product on the fibers $T_{x}\M$ and $\mathcal{U}_{x}$  correspondingly.  The space of \emph{smooth functions} on $\M$ will be denoted by $C^\infty(\M)$. The space of \emph{smooth sections} of a vector bundle $\mathcal{E}$ over $\M$ will be denoted $\Gamma^\infty(\mathcal{E})$.

\end{notation}

\begin{definition}\label{d.Foliation}
Let $\M$ be a smooth connected $n+m$-dimensional manifold. An $m$-dimensional foliation $\mathcal{F}$ on $\M$ is defined by a (maximal) collection of pairs $\{ (U_\alpha, \pi_\alpha), \alpha \in I \}$ of open subsets $U_\alpha$ of $\M$ and submersions $\pi_\alpha: U_\alpha \to U_\alpha^0$ onto open subsets of $\mathbb{R}^n$ satisfying
\begin{itemize}
\item $\bigcup_{\alpha \in I} U_\alpha =\M$;
\item If $U_\alpha \cap U_\beta \neq \emptyset$, there exists a local diffeomorphism $\Psi_{\alpha \beta}$ of $\mathbb{R}^n$ such that $\pi_\alpha=\Psi_{\alpha \beta} \pi_\beta$ on $U_\alpha \cap U_\beta $.
\end{itemize}
\end{definition}
In addition, we assume that the foliation  $\mathcal{F}$ on $\mathbb{M}$ is a Riemannian foliation with a bundle-like  metric $g$ and totally geodesic $m$-dimensional leaves. Informally a bundle-like  metric  is similar to a product metric locally, and the notion has been introduced in \cite{Reinhart1959a}. We refer to \cite{BaudoinEMS2014, MolinoBook1988, Reinhart1959a, TondeaurBook1988} for details about the geometry of Riemannian foliations, but for convenience of the reader we recall some basic definitions.

The maps $\pi_\alpha$ are called \emph{disintegrating maps} of $\mathcal{F}$. The connected components of the sets $\pi_\alpha^{-1}(c)$, $c \in \mathbb{R}^n$, are called the \emph{plaques} of the foliation. For each $p\in U_{\alpha}$, we define $\mathcal V_p:=Ker( (\pi_{\alpha})_{*p})$.
The subbundle $\mathcal{V}$ of $T\M$ with fibers $\mathcal V_p$ is referred to as the \emph{vertical distribution}. These are the vectors tangent to the leaves, the maximal integral sub-manifolds of $\mathcal{V}$.

\begin{definition}
Let $\M$ be a smooth connected $n+m$-dimensional Riemannian manifold. An $m$-dimensional foliation $\mathcal{F}$ on $\M$ is said to be \emph{Riemannian with a bundle-like metric} if the disintegrating maps $\pi_\alpha$ are Riemannian submersions onto $U_\alpha^0$ with its given Riemannian structure. If moreover the leaves are totally geodesic sub-manifolds of $\M$, then we say that the Riemannian foliation is \emph{totally geodesic with a bundle-like metric}.
\end{definition}

\subsection{Horizontal and vertical subbundles of $T\M$ and forms}\label{ss.HorizVertSubbundles}

The subbundle $\mathcal{H}$ which is normal to the vertical subbundle $\mathcal{V}$ is referred to as the set of \emph{horizontal directions}. Though this assumption is not strictly necessary in many parts of the paper, to simplify the presentation we always assume that  $\mathcal{H}$ is bracket-generating, that is, the Lie algebra of vector fields generated by global $C^{\infty}$--sections of $\mathcal{H}$ has the full rank at each point in $\M$. Using Notation \ref{n.geometry}, we denote the restrictions of the metric $g$ to $\mathcal{H}$ and $\mathcal{V}$ by $g_{\mathcal{H}}\left( \cdot, \cdot \right)$ and $g_{\mathcal{V}}\left( \cdot, \cdot \right)$ respectively.

We say that a one-form is \emph{horizontal (resp. vertical)} if it vanishes on the vertical bundle $\mathcal{V}$ (resp. on the horizontal bundle $\mathcal{H}$). Then the splitting of the tangent space
 \[
 T_x \bM= \mathcal{H}_x \oplus \mathcal{V}_x
 \]
induces a splitting of the cotangent space
 \[
 T^{\ast}_x \bM= \mathcal{H}^{\ast}_x \oplus \mathcal{V}^{\ast}_x.
 \]

The subbundle $\mathcal{H}^{\ast}$ of the cotangent bundle will be referred to as the cohorizontal bundle. Similarly, $\mathcal{V}^{\ast}$  will be referred to as the covertical bundle.

\subsection{Examples}

\begin{example}[Riemannian submersions, Hopf fibrations]\label{ex.RiemSubmersion}

Let $(\M, g)$  and  $(\B,j)$ be two smooth and connected Riemannian manifolds. A smooth surjective map $\pi: \M \to \B$ is called a \emph{Riemannian submersion} if for every $x \in \M$ the differential $T_x\pi: T_x \M \to T_{\pi(x)} \B$ is an orthogonal projection, i.e.  the map $ T_{x} \pi (T_{x} \pi)^*: T_{\pi(x)}  \B \to T_{\pi(x)} \B$
is the identity map. The foliation given by the fibers of a Riemannian submersion is then bundle-like (see  \cite[Section 2.3]{BaudoinEMS2014}). We refer to \cite[Chapter 9, Section F, pp. 249-252]{Besse1987} for Riemannian submersions with totally geodesic fibers.

The generalized Hopf fibrations (e.g. \cite[Chapter 9, Section H]{Besse1987}, \cite[Section 1.4.6]{PetersenBook3rdEdition}) offer a wide range of examples of Riemannian submersions whose fibers are totally geodesic. Let $G$ be a Lie group, and $H, K$ be two compact subgroups of $G$ with $K \subset H$. Then, we have a natural fibration given by the coset map

\begin{align*}
\pi: &  G / K   \longrightarrow   G / H
\\
& \alpha K  \longmapsto  \alpha H,
\end{align*}
where the fiber is $H/K$. From \cite{Berard-Bergery1975}, there exist $G$-invariant metrics on respectively  $G / K $ and  $ G / H $ that make $\pi$ a Riemannian submersion with totally geodesic fibers isometric to $H / K$. For instance with $G=SU(n+1)$, $H=S(U(1)  U(n)) \simeq U(n)$ and $K=SU(n)$, one obtains the usual Hopf fibration $\pi: S^{2n+1} \to \mathbb{C}P^n$, see  \cite[Chapter 9, Section H, Example 9.81]{Besse1987}. For $n=1$, this reduces to the Hopf fibration $\pi: SU(2) \simeq S^{3}  \to \mathbb{C}P^1 \simeq S^2$.

\end{example}

\begin{example}[$K$-contact manifolds]\label{ex.Kcontact}
Another important example of a Riemannian foliation is obtained in the context of contact manifolds. Let $\left( \M, \theta \right)$ be a $2n+1$-dimensional smooth contact manifold, where $\theta$ is a contact form. Then there is a unique smooth vector field $Z$ on $\M$, called the \emph{Reeb vector field}, satisfying
\[
\theta(Z)=1,\quad \mathcal{L}_Z(\theta)=0,
\]
where $\mathcal{L}_Z$ denotes the Lie derivative with respect to  $Z$. The Reeb vector field induces a foliation on $\M$, the Reeb foliation, whose leaves are the orbits of the vector field $Z$.  It is known (see \cite{Sasaki1960, Tanno1989}), that it is always possible to find a Riemannian metric $g$ and a $(1,1)$-tensor field $J$ on $\M$ so that for every  vector fields $X, Y$
\[
g(X,Z)=\theta(X),\quad J^2(X)=-X+\theta (X) Z, \quad g(X,JY)=(d\theta)(X,Y).
\]
The triple $(\M, \theta,g)$ is called a \emph{contact Riemannian manifold}.
We see then that the Reeb foliation is totally geodesic with a bundle-like metric if and only if the Reeb vector field $Z$ is a Killing field, that is,
\[
\mathcal{L}_Z g=0,
\]
as is stated in \cite[Proposition 6.4.8]{BoyerGalickiBook}. In this case, $(\M, \theta, g)$ is called a \emph{K-contact Riemannian manifold}. Observe that the horizontal distribution $\mathcal{H}$ is then the kernel of $\theta$ and that $\mathcal{H}$ is bracket generating
because $\theta$ is a contact form. We refer to \cite{BaudoinWang2014a, Tanno1989} for further details on this class of examples.
\end{example}
%
%
%
%

\subsection{Bott connection}\label{Bott connection}

If we view $(\M,g)$ as a Riemannian manifold, the Levi-Civita connection $\nabla^R$ is a natural choice for stochastic analysis on $\M$. But this connection is not adapted to the study of foliations because the horizontal and  vertical bundles may not be parallel with respect to $\nabla^R$. We will rather make use of the \emph{Bott connection} on $\mathbb{M}$ which is defined as follows.

\[
\nabla_X Y =
\begin{cases}
\pi_{\mathcal{H}} ( \nabla_X^R Y), X, Y \in \Gamma^\infty(\mathcal{H}),
\\
\pi_{\mathcal{H}} ( [X,Y]),  X \in \Gamma^\infty(\mathcal{V}), Y \in \Gamma^\infty(\mathcal{H}),
\\
\pi_{\mathcal{V}} ( [X,Y]),  X \in \Gamma^\infty(\mathcal{H}), Y \in \Gamma^\infty(\mathcal{V}),
\\
\pi_{\mathcal{V}} ( \nabla_X^R Y), X, Y \in \Gamma^\infty(\mathcal{V}),
\end{cases}
\]
where $\pi_\mathcal{H}$ (resp. $\pi_\mathcal{V}$) is the projection on $\mathcal{H}$ (resp. $\mathcal{V}$). One can check that since the foliation is bundle-like and totally geodesic the Bott connection is metric-compatible, that is, $\nabla g=0$, though unlike the Levi-Civita connection it is not torsion-free. The following properties of the Bott connection are standard but require tedious computations. We refer to \cite[Chapter 5]{TondeaurBook1988} for some of these, and  to \cite{MolinoBlog} for the details of the statements below and also point out that the Bott connection is a special case of a general class of connections introduced by R. Hladky in \cite[Lemma 2.13]{Hladky2012}.


Let  $T$ be the torsion of the Bott connection $\nabla$. Observe that for $X, Y \in  \Gamma^\infty (\mathcal{H})$

\begin{align*}
T(X,Y)&=\nabla_X Y-\nabla_Y X -[X,Y] \\
 &  =\pi_{\mathcal{H}} ( \nabla^R_X Y-\nabla^R_Y X) -[X,Y] \\
 &=\pi_{\mathcal{H}} ([X,Y]) -[X,Y] \\
 &=-\pi_{\mathcal{V}} ( [X,Y]).
\end{align*}

Similarly one can check that the Bott connection satisfies the following properties that we record here for later use
\begin{align}\label{e.1.2}
& \nabla_{X} Y \in \Gamma^\infty(\mathcal{H}) \text{ for any } X, Y \in  \Gamma^\infty(\mathcal{H}), \notag
\\
& \nabla_{X} Y \in \Gamma^\infty(\mathcal{V}) \text{ for any } X, Y \in  \Gamma^\infty(\mathcal{V}), \notag
\\
& T\left( X, Y \right) \in \Gamma^\infty(\mathcal{V}) \text{ for any } X, Y \in  \Gamma^\infty(\mathcal{H}),
\\
& T\left( U, V \right) =0 \text{ for any } U, V \in  \Gamma^\infty(\mathcal{V}), \notag
\\
& T\left( X, U \right)=0  \text{ for any } X \in \Gamma^\infty(\mathcal{H}),  U \in \Gamma^\infty(\mathcal{V}).\notag
\end{align}


\begin{example}[Example \ref{ex.RiemSubmersion} revisited]\label{connection submersion}
Let $\pi: (\M, g)\to (\B,j)$ be a Riemannian submersion with totally geodesic leaves. A vector field $X \in \Gamma^\infty(T\M)$ is said to be \emph{projectable} if there exists a smooth vector field $\overline{X}$ on $\B$ such that for every $x \in \M$,  $T_x \pi ( X(x))= \overline {X} (\pi (x))$. In that case, we say that $X$ and $\overline{X}$ are $\pi$-related. A vector field $X$ on $\M$ is called \emph{basic} if it is projectable and horizontal. If  $\overline{X}$ is a smooth vector field on $\B$, then there exists a unique basic vector field $X$ on $\M$ which is $\pi$-related to $\overline{X}$. This vector is called the \emph{lift} of $\overline{X}$.
The Bott connection is then a lift of the Levi-Civita connection of $(\B,j)$ in the following sense:
\begin{align}\label{projectsdown}
\nabla^{\B}_{\overline{X}} \overline{Y} =\overline{ \nabla_X Y}, \quad \overline{X}, \overline{Y} \in \Gamma^\infty(\mathbb{B}),
\end{align}
where $\nabla^\B$ is the Levi-Civita connection on $\B$.
\end{example}

\begin{example}[Example \ref{ex.Kcontact} revisited]
Let $(\M, \theta,g)$ be a K-contact Riemannian manifold. The Bott connection coincides  with Tanno's connection that was introduced in \cite{Tanno1989} and which is the unique connection that satisfies the following properties.
\begin{enumerate}
\item $\nabla\theta=0$;
\item $\nabla Z=0$;
\item $\nabla g=0$;
\item ${T}(X,Y)=d\theta(X,Y)Z$ for any $X,Y\in \Gamma^\infty(\mathcal{H})$;
\item ${T}(Z,X)=0$ for any vector field $X\in \Gamma^\infty(\mathcal{H})$.
\end{enumerate}
\end{example}

\subsection{Orthonormal frame bundle}\label{s.OFB} We will use standard notation for orthonormal frame bundles. Suppose $\M$ is a compact Riemannian manifold of dimension $d$. Note that in the setting of Riemannian foliations we have $d=n+m$. Recall that a frame at $x \in \M$ can be described as a linear isomorphism $u: \mathbb{R}^{d} \to T_x\M$ such that for the standard basis $\left\{  e_{i} \right\}_{i=1}^{d}$ of $\mathbb{R}^{d}$ the collection $\left\{ u\left( e_{i} \right)\right\}_{i=1}^{d}$ is a basis (frame) for $T_{x}\M$. The collection of all such frames $\mathcal{F}\left( \M \right):=\bigcup_{x \in \M} \mathcal{F}\left( \M \right)_{x}$ is called the \emph{frame bundle} with the group $\operatorname{GL}\left( \mathbb{R}, d \right)$ acting on the bundle. If $\M$ is in addition Riemannian, we can restrict ourselves to consideration of Euclidean isometries  $u: \left( \mathbb{R}^{d}, \langle \cdot,\cdot \rangle \right) \to \left( T_x\M, g \right)$ with the group $\operatorname{O}\left( \mathbb{R}, d \right)$ acting on the bundle. The \emph{orthonormal frame bundle} will be denoted by $\mathcal{O}(\M)$.

Suppose that $D$ is a connection on $\M$, then $D$ induces a decomposition of each tangent space $T_{u}\mathcal{O}(\M)$ into the direct sum of a horizontal subspace and a vertical subspace as described in \cite[Section 2.1]{HsuEltonBook}. Using such decomposition, one can then lift smooth maps on $\M$ into smooth horizontal paths on $\mathcal{O}(\M)$, see \cite[p. 421]{Hsu1995b}. Such a lift is usually called the horizontal lift to $\mathcal{O}(\M)$. However, to avoid the confusion with the notion of horizontality given by the foliation on $\M$, in this paper it shall often simply be referred to as the $D$-lift to $\mathcal{O}(\M)$.


\section{Horizontal calculus of variations}\label{s.HorizontCalculus}


To motivate the definition of the tangent processes to the horizontal Brownian motion on $\M$ that we will use to prove quasi-invariance, we first present results on the horizontal calculus of variations of deterministic paths.

\subsection{Adapted connections}

Using the notation in Section \ref{s.OFB}, we consider $u \in \mathcal{O} (\M)$.
To take into account the foliation structure on $\M$, we shall be interested in a special subbundle of $\mathcal{O} (\M)$, the \textit{horizontal frame bundle}.

\begin{definition}
An isometry $u: (\mathbb{R}^{n+m}, \langle \cdot,\cdot \rangle ) \longrightarrow (T_x\M, g)$ will be called \emph{horizontal} if $u ( \mathbb{R}^n \times \{ 0 \}) \subset \mathcal{H}_x$ and $u (  \{ 0 \}  \times \mathbb{R}^m ) \subset \mathcal{V}_x$. The \emph{horizontal frame bundle} $\mathcal{O}_\mathcal{H} (\M)$ is then defined as the set of $(x,u) \in \mathcal{O} (\M)$ such that $u$ is horizontal.
\end{definition}

For notational convenience, when needed we identify $\mathbb{R}^{n+m}$ with  $\mathbb{R}^n\times \mathbb{R}^m$, hence we have embeddings of $\mathbb{R}^n$ and $\mathbb{R}^m$ into  $\mathbb{R}^{n+m}$.


\begin{assumption}\label{property D}\leavevmode
We assume that $D$ is a connection on $\M$ satisfying the following properties.

\begin{itemize}
\item $D$ is a \emph{metric connection} on $\M$, that is, $D g =0$;

\item  $D$ is \emph{adapted to the foliation} $\mathcal{F}$ in the following sense
\begin{align*}
& D_{X} Y \in \Gamma^\infty(\mathcal{H}), \text{ if } X \in \Gamma^\infty(\M), Y \in \Gamma^\infty(\mathcal{H}),
\\
& D_{X} Z \in \Gamma^\infty(\mathcal{V}), \text{ if }  X \in \Gamma^\infty(\M), Z \in \Gamma^\infty(\mathcal{V});
\end{align*}

\item For every $X \in \Gamma^\infty(\mathcal{H}), Y \in \Gamma^\infty(\M) $, $D_X Y =\nabla_X Y$, where $\nabla$ is the Bott connection.
\end{itemize}
\end{assumption}

\begin{remark}\label{projection connection}
In the case of a Riemannian submersion in Example \ref{ex.RiemSubmersion}, these assumptions imply that  the connection $D$ is a lift of the Levi-Civita connection on $(\B, j)$, namely,
\begin{align*}
\nabla^{\B}_{\overline{X}} \overline{Y} =\overline{ D_X Y}, \quad \overline{X}, \overline{Y} \in \Gamma^\infty(\B),
\end{align*}
where $\nabla^\B$ is the Levi-Civita connection on $\B$. We refer to Example \ref{connection submersion} for further details.
\end{remark}
Of course, an example of a connection $D$ that satisfies the above assumptions is given by the Bott connection $\nabla$ itself. However, we state the results of the section in  greater generality using a connection $D$ satisfying Assumption \ref{property D}. This generality  is relevant for Section \ref{p.2}, where we use other connections than the Bott connection (see Remark \eqref{compatibility F}). The main reason for using different connections is that while the Bott connection is adapted to the foliation structure, the torsion of the Bott connection is not skew-symmetric. 

The connection $D$ allows us to lift vector fields on $\M$ to vector fields on $\mathcal{O} (\M)$ (see \cite[p.421]{Hsu1995b}). Let $e_1, \cdots, e_n, f_1, \cdots, f_m$ be the standard basis of $\mathbb{R}^{n+m}$.

\begin{notation}\label{HVframe}
We denote by $A_i$ the vector field on $\mathcal{O} (\M)$ such that   $A_i (x,u)$ is the lift of $u(e_i)$, $i=1, ..., n$, $\left( x, u \right) \in \mathcal{O} (\M)$, and we denote by $V_{j}$ the vector field on $\mathcal{O} (\M)$ such that $V_{j} (x,u)$ is the lift of $u(f_j)$, $j=1, .., m$. We sometimes call $A$ and $V$ \emph{fundamental vector fields} on $\mathcal{O} (\M)$. For any $v \in \mathbb{R}^{n+m}$, we denote
\begin{align*}
& Av:=\sum_{i=1}^n v_i A_i,
\\
& Vv :=\sum_{j=1}^m v_{j+n} V_j.
\end{align*}
Then $Av$ and $Vv$ are  vector fields on $\mathcal{O}(\M)$ whose values at some $u \in \mathcal{O}(\M)$ will be denoted  respectively by $A_u v$ and $V_uv$.
\end{notation}

\begin{notation} Let $x_{0}$ be a fixed point in $\M$.
By $W_0^\infty(\mathbb{R}^{n+m})$ we denote the space of smooth paths $v : [0,1] \longrightarrow \mathbb{R}^{n+m}$ such that $v(0)=0$, and by $W_{x_0}^\infty(\mathbb{M})$ we denote the space of smooth paths $\gamma : [0,1] \longrightarrow  \mathbb{M}$ such that $\gamma(0)=x_0$.
\end{notation}

\subsection{Development maps}

Next we would like to recall the notion of a rolling map $\phi$ between path spaces over $\M$ and $\mathbb{R}^{n+m}$ or equivalently development and anti-development maps (see for instance \cite[Section 2]{Hsu1995b}). Let $\pi: \mathcal{O}(\M) \rightarrow \M$ be the bundle  projection map. To define the rolling map $\phi: W_0^\infty(\mathbb{R}^{n+m}) \rightarrow  W_{x_0}^\infty(\mathbb{M})$ we need the following  differential equation on $\mathcal{O}(\M)$

\begin{equation}\label{e.RM}
du_{s} =\sum_{i=1}^n A_i (u_{s})  d\omega^{i}_{s} +\sum_{i=1}^m V_i(u_{s})d\omega^{n+i}_{s}=A_{u_{s}} d\omega_{s} +V_{u_{s}} d\omega_{s},
\end{equation}
where $\omega \in W_0^\infty(\mathbb{R}^{n+m})$. By compactness of $\M$ and thus of $\mathcal{O}(\M)$ this equation has a unique solution given an initial condition $u_0 \in \mathcal{O}(\M)$. In the sequel we fix  $u_0 \in \mathcal{O}(\M)$ such that $\pi (u_0)=x_0$.

\begin{definition}{ \ }

\label{develop}

\begin{enumerate}
\item For any $\omega \in W_0^\infty(\mathbb{R}^{n+m})$  the  \emph{development} of $\omega$ in $\M$ is defined as $\gamma_{s}=\pi(u_{s})$, where $\left\{ u_{s} \right\}_{s \in [0,1]}$ is the solution to \eqref{e.RM} with initial condition $u_0$. Then we denote $\phi (\omega):=\gamma$. The map $\phi$ is also called the \emph{rolling map}.
\item For any $\gamma \in W_{x_0}^\infty(\mathbb{M})$ the  \emph{anti-development} of $\gamma$ in $\mathbb{R}^{n+m}$   is the unique path $\omega  \in W_0^\infty  (\mathbb{R}^{n+m})$ such that if  $\left\{ u_{s} \right\}_{s \in [0,1]}$ is the solution to  \eqref{e.RM}, then $\gamma_{s}=\pi(u_{s})$. Then we  denote  $\phi^{-1} (\gamma):=\omega$.
\end{enumerate}

\end{definition}
This definition extends to continuous semimartingales, in which case we speak of stochastic development and stochastic anti-development (e.g.  \cite[Section 2.3]{HsuEltonBook} and \cite[p. 433]{Hsu1995b}).

\subsection{Horizontal paths}

\begin{definition}
A smooth path $\omega: [0,1] \to \mathbb{R}^{n+m}$ is called \emph{horizontal} if it takes values in $\mathbb{R}^n$.  The space of smooth horizontal paths such that $\omega(0)=0$ will be denoted by $W^\infty_{0,\mathcal{H}} (\mathbb{R}^{n+m})$.
\end{definition}

\begin{definition}  \label{d.HorizontalPaths}
A smooth path $\gamma: [0,1] \to \mathbb{M}$ is called \emph{horizontal} if for every vertical smooth one-form $\theta$ we have $\int_\gamma \theta =0$.  The space of smooth horizontal paths such that $\gamma(0)=x_0$ will be denoted $W^\infty_{x_0,\mathcal{H}} (\M)$.
\end{definition}

\begin{remark}
The space $W^\infty_{x_0,\mathcal{H}} (\M)$ contains only smooth paths, therefore it can be equivalently described as follows.  A path $\gamma$ is in $W^\infty_{x_0,\mathcal{H}} (\M)$ if and only if $\gamma^{\prime}(s) \in \mathcal{H}_{\gamma(s)}$ for every $s \in [0,1]$. The advantage of Definition \ref{d.HorizontalPaths} is that it will easily extend to non-smooth paths such as  semimartingales.
\end{remark}
The next step is to define the horizontal rolling map $\phi_\mathcal{H}: W^\infty_{0,\mathcal{H}} (\mathbb{R}^{n+m}) \rightarrow W^\infty_{x_0,\mathcal{H}} (\M)$ similarly to Definition \ref{develop} on the spaces of horizontal paths. For any $\omega^\mathcal{H} \in W^\infty_{0,\mathcal{H}} (\mathbb{R}^{n+m})$ we consider the  differential equation on $\mathcal{O}(\M)$ with initial condition $u_0$

\begin{equation}\label{e.HorizontalRM}
du_{s} =\sum_{i=1}^n A_i (u_{s})  d\omega^{\mathcal{H}, i}_{s}=A_{u_{t}}  d\omega^{\mathcal{H}}_{s}.
\end{equation}
Observe that for $\gamma=\pi(u)$ we have
\[
d\gamma_{s} =\sum_{i=1}^n d\pi (A_i (u_{s}) ) d\omega^{\mathcal{H}, i}_{s},
\]
and therefore $\gamma$ is horizontal since $d\pi(A_i(u_{s}))$ is.

\begin{lem} \label{horizontal development map}
Suppose $\gamma \in W^\infty_{x_0,\mathcal{H}} (\M)$, then there exists a unique $\omega^\mathcal{H}  \in W^\infty_{0,\mathcal{H}} (\mathbb{R}^{n+m})$ such that if  $\left\{ u_{s} \right\}_{s \in [0,1]}$ is the solution to \eqref{e.HorizontalRM}, then $\gamma_{s}=\pi(u_{s})$.
\end{lem}

\begin{proof}
As before, let $e_1, \cdots, e_n, f_1, \cdots, f_m$ be the standard basis of $\mathbb{R}^{n+m}$.
Note that any $\gamma \in W^\infty_{x_0,\mathcal{H}} (\M)$  can be viewed as an element in $W_{x_0}^\infty(\mathbb{M})$. Let $\omega  \in W_0^\infty  (\mathbb{R}^{n+m})$ be the anti-development of $\gamma$ introduced in Definition \ref{develop}. Then if  $\left\{ u_{s} \right\}_{s \in [0,1]}$ is the solution to the  differential equation \eqref{e.RM} with initial condition $u_0$, then $\gamma_{s}=\pi(u_{s})$.  Since $\gamma$ is horizontal,  then for every smooth vertical one-form $\theta$ one has
\[
\int_{\gamma[0, s]} \theta=0.
\]
Therefore
\[
\int_{\gamma[0, s]} \theta=\sum_{i=1}^n \int_0^s \theta (u_r e_i ) d\omega^{i}_r+\sum_{i=1}^m \int_0^s  \theta (u_r f_i) d\omega^{n+i}_r=0.
\]
The form $\theta$ being vertical, one deduces
\[
\sum_{i=1}^m \int_0^t  \theta (u_s f_i) d\omega^{n+i}_s=0.
\]
Since it is true for any $\theta$, one deduces
\[
\sum_{i=1}^m \int_0^s  (u_r f_i) d\omega^{n+i}_r=0.
\]
Now observe that $u_r f_1, \cdots, u_sr f_m$ are linearly independent, thus for every $r$ one has $d\omega^{n+i}_r=0$. As a conclusion, $\omega$ is horizontal.

\end{proof}

\begin{definition}[Horizontal development and anti-development]\label{d.HorizSmoothItoMaps} { \ }

\begin{enumerate}
\item For any $\omega^\mathcal{H} \in W^\infty_{0,\mathcal{H}} (\mathbb{R}^{n+m})$  the  \emph{horizontal development} of $\omega$ in $\M$ is $\gamma_{s}=\pi(u_{s})$, where $\left\{ u_{s} \right\}_{s \in [0,1]}$ is the solution to \eqref{e.HorizontalRM} with initial condition $u_0 \in \mathcal{O}(\M)$. Then we denote $\phi_\mathcal{H} (\omega):=\gamma$. The map $\phi_\mathcal{H} $ is called the \emph{horizontal rolling map}.
\item For any $\gamma \in W^\infty_{x_0,\mathcal{H}} (\M)$ the  \emph{horizontal anti-development} of $\gamma$  is $\omega^\mathcal{H}  \in W^\infty_{0,\mathcal{H}} (\mathbb{R}^{n+m})$ is the unique path such that if  $\left\{ u_{s} \right\}_{s \in [0,1]}$ is the solution to  \eqref{e.HorizontalRM} with initial condition $u_0$, then $\gamma_{s}=\pi(u_{s})$. Then we  denote  $\phi_{\mathcal{H}}^{-1} (\gamma):=\omega$.  
\end{enumerate}
\end{definition}

\subsection{Paths tangent to horizontal paths}
For any $v \in W_0^\infty(\mathbb{R}^{n+m})$ we consider the vector field $\mathbf D_v$ on $W_{x_0}^\infty(\mathbb{M})$ defined by
\[
\mathbf D_v (\gamma)_s:=u_s(\gamma) v_s, \hskip0.1in \gamma \in W^\infty (\M),
\]
where $u$ is the  $D$-lift of $\gamma$ to $\mathcal{O}(\M)$.
Let $\{ \zeta_t^v, t\in\R \}$ be the flow generated by $\mathbf D_v$, i.e.
\[
\frac{d}{dt} ( \zeta_t^v \gamma)_s=\mathbf D_v ( \zeta_t^v \gamma)_s, \quad \zeta_0^v \gamma=\gamma.
\]
One can use the development and anti-development maps $\phi$ and $\phi^{-1}$ in Definition \ref{develop} to introduce a flow on $W_0^\infty(\mathbb{R}^{n+m})$ as follows
\[
\xi_t^v :=\phi^{-1} \circ  \zeta_t^v \circ \phi, \quad t \in \R.
\]
Note that $\mathbf D_v$, $\zeta_t^v$ and $\xi_t^v$  depend on the connection $D$. We now recall \cite[Theorem 2.1]{Hsu1995b} that describes the generator of the flow $\xi_t^v$ in the situation when a connection is metric-compatible but not necessarily torsion-free.

\begin{thm}[Theorem 2.1 in \cite{Hsu1995b}] \label{Hsu deterministic}
Suppose that $v \in W_0^\infty(\mathbb{R}^{n+m})$  and $\omega \in W_0^\infty(\mathbb{R}^{n+m})$. Then
\[
\left.\frac{d}{dt}\right|_{ t=0} \xi_t^v  (\omega)_s =p_v(\omega)_{s},
\]
where
\begin{align*}
 p_v(\omega)_{s} & = v(s)-\int_0^s T^D_{u_r} (  A d\omega_r+Vd\omega_r,  Av(r)+Vv(r)) -
 \\
 & \int_0^s \left(  \int_0^r \Omega^D_{u_\tau}( A   d \omega_\tau +Vd \omega_\tau, A v(\tau)+V v (\tau)) \right)   d \omega_r.
\end{align*}
Here $u$ is the  $D$-lift to $\mathcal{O}(\M)$ of the development of $\omega$,  $T^D$ is the torsion form of the connection $D$ and  $\Omega^D$ is its curvature  form.
\end{thm}
We are interested in the variation of horizontal paths. Let us observe that for $\omega^\mathcal{H} \in W^\infty_{0,\mathcal{H}} (\mathbb{R}^{n+m})$
\begin{align}\label{pullback1}
 p_v(\omega^\mathcal{H})_{s} & = v(s)-\int_0^s T^D_{u_r} (  A  d\omega^\mathcal{H}_r ,  Av(r)+Vv(r))
 \\
 & -\int_0^s \left(  \int_0^r \Omega^D_{u_\tau}( A  d \omega^\mathcal{H}_\tau , A v(\tau)+V v (\tau)) \right)   d \omega^\mathcal{H}_r. \notag
\end{align}

\begin{definition}\label{d.horizontnagent}
We will say that $v \in W_0^\infty(\mathbb{R}^{n+m})$ is \emph{tangent to the horizontal path}  $\gamma \in W^\infty_{x_0,\mathcal{H}} (\M)$ if for every $s \in [0,1]$, $\left. \frac{d}{dt}\right|_{ t=0} \phi^{-1} (\zeta_t^v \gamma)_s \in \mathbb{R}^n$.
\end{definition}

\begin{remark}
From this definition, $v \in W_0^\infty(\mathbb{R}^{n+m})$ is tangent to the horizontal path  $\gamma$  if and only if  $p_v(\omega^\mathcal{H})$ is horizontal, where $\omega$ is the horizontal anti-development of $\gamma$. Intuitively, $v$ is tangent to $\gamma$ if it yields a variation of $\gamma$ in the horizontal directions only. More precisely, call a vector field $\xi$ along $\gamma \in W^\infty_{x_0,\mathcal{H}} (\M)$ an \textit{horizontal variation} of $\gamma$ if $\xi(x_0)=0$ and if there exists $(\sigma_t)_{t\in [-\ve,\ve] } \subset W^\infty_{x_0,\mathcal{H}} (\M)$ with $\sigma_0=\gamma$ such that $\left.\frac{d}{dt}\right|_{ t=0} (\sigma_t)_s =\xi_s$ for $s \in [0,1]$. Then, by Theorem \ref{Hsu deterministic} and Proposition \ref{v2jk},   $\xi$ is an horizontal variation of $\gamma$ if and only if $u_s(\gamma)^{-1} \xi_s$ is tangent to the horizontal path $\gamma$. Let us note that the notion of horizontal variation is independent from any metric and any connection. It therefore yields an intrinsic notion of horizontal tangent path space.  We are grateful to the referee for this observation.

\end{remark}

\begin{remark}
One should note that even if $v \in W_0^\infty(\mathbb{R}^{n+m})$ is tangent to the horizontal path  $\gamma$, it may not be true that for every $t \in \mathbb{R}$, $\zeta_t^v \gamma \in W^\infty_{x_0,\mathcal{H}} (\M)$.
\end{remark}

One  has the following characterization of tangent paths, which is the main result of the section.

\begin{thm}\label{tangent deterministic}
Let $\gamma \in W^\infty_{x_0,\mathcal{H}} (\M)$.  A path $v \in W_0^\infty(\mathbb{R}^{n+m})$ is tangent to the horizontal path  $\gamma$ if and only if the path
 \[
 v(s)-\int_0^s T_{u_r} (  A d\omega^\mathcal{H}_r ,  Av(r))
 \]
is horizontal, i.e. takes values in $\mathbb{R}^n$, where $\omega^\mathcal{H} $ is the horizontal anti-development of $\gamma$, $u$ is its $D$-lift to $\mathcal{O}(\M)$, and $T$ is  the torsion  of the Bott connection.
\end{thm}

\begin{proof}
The path $v \in W_0^\infty(\mathbb{R}^{n+m})$ is tangent to the horizontal path  $\gamma$ if and only if the path
\begin{align*}
 p_v(\omega^\mathcal{H})_{s} & = v(s)-\int_0^s T^D_{u_s} (  A d\omega^\mathcal{H}_r ,  Av(r)+Vv(r))  \\
 & -\int_0^s \left(  \int_0^r \Omega^D_{u_\tau}( A   d \omega^\mathcal{H}_\tau , A v(\tau)+V v (\tau)) \right)   d \omega^\mathcal{H}_r \notag
\end{align*}
is horizontal. Since $D$ satisfies Assumption \ref{property D}, the integral
\[
\int_0^s \left(  \int_0^r \Omega^D_{u_\tau}( A   d \omega^\mathcal{H}_\tau , A v(\tau)+V v (\tau)) \right)   d \omega^\mathcal{H}_r
\]
is always horizontal. Let us now denote by $J$ the difference between connections $D$ and $\nabla$, that is, the tensor $J$ is defined  for any $X, Y \in \Gamma^\infty(\M)$ by
\[
J_XY =D_X Y -\nabla_X Y.
\]
We have then
\begin{align*}
T^D(X,Y)&=D_X Y-D_Y X-[X,Y] \\
 &=T(X,Y)+J_X Y -J_Y X.
\end{align*}
Let us assume that $X$ is horizontal. We have then $J_X=0$, because $D_\mathcal{H}=\nabla_\mathcal{H}$. Also $J_Y X$ is horizontal, because $D$ is adapted to the foliation $\mathcal{F}$. We deduce that the vertical part of
\[
 v(s)-\int_0^s T^D_{u_r} (  Ad\omega^\mathcal{H}_r ,  Av(r)+Vv(r))
 \]
 is the same as the vertical part of
 \[
 v(s)-\int_0^s T_{u_r} (  A d\omega^\mathcal{H}_r ,  Av(r)+Vv(r)).
 \]
 We conclude that the vertical part of $p_v(\omega^\mathcal{H})$ is zero if and only if the vertical part of
 \[
  v(s)-\int_0^s T_{u_r} (  A d\omega^\mathcal{H}_r ,  Av(r)+Vv(r)).
 \]
 is zero. By the  properties in Equation \eqref{e.1.2},  we have
 \[
 \int_0^s T_{u_r} (  A d\omega^\mathcal{H}_r ,  Vv(r))=0,
 \]
which concludes the proof.
\end{proof}

\begin{remark}
{\ }
\label{Remark independent}
By Theorem \ref{tangent deterministic}, the notion of tangent path  does not depend on the particular choice of the connection $D$ as long as it satisfies Assumption \ref{property D}.
\end{remark}

\subsection{Variations on the horizontal path space}\label{s.HorizontVariations}

In this section, we describe two types of variations on the horizontal path space that are induced by tangent paths. The first one is explicit and inspired by the approach by B.~Driver in \cite{Driver2004a}. The second one is  based on more classical flow constructions. The key ingredient is the following lemma.

\begin{lem}\label{construction tangent}
Let $h \in W_{0,\mathcal{H}}^\infty(\mathbb{R}^{n+m})$. If $\omega^\mathcal{H} \in W_{0,\mathcal{H}}^\infty(\mathbb{R}^{n+m})$, then
\begin{align}\label{j'en ai marre}
\tau_h (\omega^\mathcal{H})_{s}=h(s)+\int_0^s T_{u_r} (  A d\omega^\mathcal{H}_r ,  Ah(r))
\end{align}
is a tangent path to $\phi(\omega^\mathcal{H})$, where $u$ denotes the $D$-lift of the horizontal development of $\omega^\mathcal{H}$.
\end{lem}
\begin{proof}
Let
\[
v(s)=h(s)+\int_0^s T_{u_r} (  A d\omega^\mathcal{H}_r ,  Ah(r)).
\]
Since $h$ is horizontal and $T$ is a vertical tensor, one deduces that the horizontal part of $v$ is $h$. Therefore,
\[
v(s) -\int_0^s T_{u_r} (  Ad\omega^\mathcal{H}_r ,  Av(r))=h(s)
\]
is horizontal.
\end{proof}

Let $v \in W_0^\infty(\mathbb{R}^{n+m})$, $\omega^\mathcal{H} \in W_{0,\mathcal{H}}^\infty(\mathbb{R}^{n+m})$ and assume that $v$ is tangent to the horizontal development of $\omega^\mathcal{H}$. Recall that
\begin{align*}
 p_v(\omega^\mathcal{H})_{s} & = v(s)-\int_0^s T^D_{u_s} (  A d\omega^\mathcal{H}_r ,  Av(r)+Vv(r))
 \\
 & -\int_0^s \left(  \int_0^r \Omega^D_{u_\tau}( A   d \omega^\mathcal{H}_\tau , A v(\tau)+V v (\tau)) \right)   d \omega^\mathcal{H}_r. \notag
\end{align*}
As before,  let us now denote by $J$ the difference between connections $D$ and $\nabla$. For $X,Y \in \Gamma^\infty(\M)$, we have thus
\[
J_XY =D_X Y -\nabla_X Y.
\]
We can then write
\begin{align*}
p_v(\omega^\mathcal{H})_{s}  &= v_\mathcal{H}(s)+\int_0^s (J_{Vv(r) })_{u_r} (  A d\omega^\mathcal{H}_r )
\\
& -  \int_0^s \left(  \int_0^r \Omega^D_{u_\tau}( A   d \omega^\mathcal{H}_\tau , A v(\tau)+V v (\tau)) \right)   d \omega^\mathcal{H}_r. \notag
\end{align*}

More concisely, we have therefore
\[
p_v (\omega^\mathcal{H})_s =v_\mathcal{H} (s) +\int_0^s q_v(\omega^\mathcal{H})_u d\omega^\mathcal{H}_u,
\]
where $q_v(\omega^\mathcal{H})_u  \in \mathfrak{so}(n)$ is defined in such a way that
\begin{align*}
 & \int_0^s q_v(\omega^\mathcal{H})_u d\omega^\mathcal{H}_u  \\
 =&\int_0^s (J_{Vv(r) })_{u_r} (  A d\omega^\mathcal{H}_r ) - \int_0^s \left(  \int_0^r \Omega^D_{u_\tau}( A   d \omega^\mathcal{H}_\tau , A v(\tau)+V v (\tau)) \right)   d \omega^\mathcal{H}_r.
\end{align*}
As a consequence, with the above notation, one has that for every $h \in W_{0,\mathcal{H}}^\infty(\mathbb{R}^{n+m})$
\[
p_{\tau_h(\omega^\mathcal{H})}  (\omega^\mathcal{H})_s =h (s)+\int_0^s q_{_{\tau_h(\omega^\mathcal{H})}}(\omega^\mathcal{H})_u d\omega^\mathcal{H}_u ,
\]
We are now ready to introduce two relevant variations of horizontal paths.
\begin{notation}
Let $h \in W_{0,\mathcal{H}}^\infty(\mathbb{R}^{n+m})$.
\begin{enumerate}

\item For $t \in \mathbb{R} $, we define  a map $\rho_t^h: W_{0,\mathcal{H}}^\infty(\mathbb{R}^{n+m}) \to W_{0,\mathcal{H}}^\infty(\mathbb{R}^{n+m})$ by
\begin{align}\label{variation tangent}
(\rho_t^h \omega^\mathcal{H})_s :=\int_0^s e^{t q_{\tau_h (\omega^\mathcal{H})}(\omega^\mathcal{H})_u} d\omega^\mathcal{H}_u+t h(s).
\end{align}
\item For $t \in \mathbb{R} $, we define  a map $\nu_t^h: W_{0,\mathcal{H}}^\infty(\mathbb{R}^{n+m}) \to W_{0,\mathcal{H}}^\infty(\mathbb{R}^{n+m}) $ as the flow generated by $p_{\tau_h}$
\[
\frac{d}{dt} (\nu_t^h \omega^\mathcal{H})_s=p_{\tau^h(\nu_t^h \omega^\mathcal{H})}(\nu_t^h \omega^\mathcal{H})_s, \quad \nu_0^h \omega^\mathcal{H}= \omega^\mathcal{H}.
\]
\end{enumerate}

\end{notation}

\begin{remark}
Unless $q_{\tau_h}=0$,  the family $\left\{ \rho_t^h, t \in \mathbb{R} \right\}$ is \textbf{not} a flow on $W_{0,\mathcal{H}}^\infty(\mathbb{R}^{n+m})$, but it is  a convenient explicit  one-parameter variation, since we observe that $\rho_0^h \omega^\mathcal{H}= \omega^\mathcal{H}$ and
\[
\frac{d}{dt}\left|_{ t=0}  (\rho_t^h \omega^\mathcal{H})_s \right. =p_{\tau_h (\omega^\mathcal{H})}(\omega^\mathcal{H})_s.
\]
\end{remark}
We then have  the following result, which is immediate in view of Theorem \ref{Hsu deterministic} since
\[
\frac{d}{dt}\left|_{ t=0}  (\rho_t^h \omega^\mathcal{H})_s \right. =\frac{d}{dt}\left|_{ t=0}  (\nu_t^h \omega^\mathcal{H})_s \right.=p_{\tau_h (\omega^\mathcal{H})}(\omega^\mathcal{H})_s.
\]

\begin{proposition}[Variation of horizontal paths along tangent paths]\label{v2jk}

Let $h \in W_{0,\mathcal{H}}^\infty(\mathbb{R}^{n+m})$, then for every $\gamma \in W_\mathcal{H}^\infty(\M)$
\[
\left.\frac{d}{dt}\right|_{ t=0}  \phi_\mathcal{H} \circ \rho_t^h  \circ \phi_\mathcal{H}^{-1}(\gamma)_s =\left.\frac{d}{dt}\right|_{ t=0}  \phi_\mathcal{H} \circ \nu_t^h  \circ \phi_\mathcal{H}^{-1}(\gamma)_s=u_s(\gamma) \tau_h (\omega^\mathcal{H})_s,
\]
where $u$ is the $D$-lift of $\gamma$, and $\omega^\mathcal{H}$ is its horizontal development.
\end{proposition}

%

\section{Quasi-invariance of the horizontal Wiener measure}\label{p.1}

In this part of the paper, we first describe two constructions of the horizontal Brownian motion, and then we develop horizontal stochastic calculus and prove  quasi-invariance of the horizontal Wiener measure. Throughout this section we consider a smooth connected $n+m$-dimensional Riemannian manifold $\M$  equipped with the structure of an $m$-dimensional foliation $\mathcal{F}$,  a bundle-like metric $g$ and totally geodesic $m$-dimensional leaves. In addition, we assume that $\M$ is compact.

\subsection{Horizontal Brownian motion}\label{SectionBM}

\subsubsection{Construction from the horizontal Dirichlet form}\label{Dirichlet}

We define the horizontal gradient $\nabla_\mathcal{H} f$ of a smooth function $f$ as the projection of the Riemannian gradient of $f$ on the horizontal bundle $\mathcal{H}$. Similarly, we define the vertical gradient $\nabla_\mathcal{V} f$ of a function $f$ as the projection of the Riemannian gradient of $f$ on the vertical bundle $\mathcal{V}$.

Consider the pre-Dirichlet form
\[
\mathcal{E}_{\mathcal{H}} (f, h) =\int_\bM g_{\mathcal{H}}\left( \nabla_\mathcal{H} f, \nabla_\mathcal{H} h \right)\hskip0.05in  d\operatorname{Vol}, \hskip0.05in f, h \in C^\infty (\M),
\]
where $d\operatorname{Vol}$ is the Riemannian volume measure on $\M$. We note that $\mathcal{E}_{\mathcal{H}}$ is closable since it can be dominated by the Dirichlet form generated by the Laplace-Beltrami on $\M$ which is closable since $\M$ is compact, thus complete.
Then there exists a unique diffusion operator $L$ on $\M$ such that for all $f, h \in C^\infty (\M)$
\[
\mathcal{E}_{\mathcal{H}} (f, h) =-\int_\M f L h \hskip0.05in d\operatorname{Vol}=-\int_\M h Lf \hskip0.05in d\operatorname{Vol}.
\]
The operator $L$ is called the \emph{horizontal Laplacian} of the foliation. If $\left\{ X_{i} \right\}_{i=1}^{n}$ is a local orthonormal frame of horizontal vector fields, then we can write $L$ in this frame
\begin{align}\label{Hormander form}
L=\sum^n_{i=1} X_i^2 +X_0,
\end{align}
where $X_0$ is a smooth vector field.
Observe that the subbundle $\mathcal{H}$ satisfies H\"{o}rmander's (bracket generating) condition, therefore by H\"ormander's theorem the operator $L$ is locally subelliptic (for comments on this terminology introduced by Fefferman-Phong we refer to \cite{FeffermanPhong1983},  see also the survey papers   \cite{BaudoinLevico, JerisonSanchez-Calle1987} or  \cite[p. 944]{DriverGrossSaloff-Coste2009a}).

By \cite[Proposition 5.1]{BaudoinEMS2014} the completeness of the Riemannian metric $g$ implies that $L$  is essentially self-adjoint on $C^\infty(\M)$ and thus that $\mathcal{E}_\mathcal{H}$ is uniquely closable. Then we can define the semigroup $P_s=e^{\frac{s}{2} L}$ by using the spectral theorem. The diffusion process $\left\{ W_s \right\}_{s \geqslant 0}$ corresponding to the semigroup $\left\{ P_s \right\}_{s \geqslant 0}$ will be called the \emph{horizontal Brownian motion} on the Riemannian foliation $(\mathcal{F}, g)$. Since $\M$ is assumed to be compact, $1 \in \operatorname{dom} (\mathcal{E}_\mathcal{H})$ and thus $P_s 1=1$. This implies that $\left\{ W_s \right\}_{s \geqslant 0}$ is a non-explosive diffusion.

If the horizontal Laplacian can be written in the form \ref{Hormander form} globally for smooth horizontal vector fields $X_0,X_1,\cdots, X_n$, then $\left\{ W_t \right\}_{t \geqslant 0}$ can be constructed from a stochastic differential equation on $\M$.

Even if the horizontal Laplacian can not be written in the form \ref{Hormander form} globally, the horizontal Brownian motion  $\left\{ W_s \right\}_{s \geqslant 0}$ can still be constructed from a globally defined stochastic differential equation on a bundle over $\M$ (see  \cite[Theorem 3.8]{Elworthy2014} or Corollary \ref{projec}). The following section provides an explicit description of such a construction that shall be used in the sequel.

\subsubsection{Construction  from the  orthonormal frame bundle}\label{Eels-Elworthy}
We can write the vector fields $\left\{ A_i \right\}_{i=1}^{n}$ locally in terms of the normal frames introduced in \cite{BaudoinKimWang2016}.

\begin{lemma}[Lemma 2.2 in \cite{BaudoinKimWang2016}]\label{frame}
Let $x_0 \in \M$. Around $x_0$, there exist a local orthonormal horizontal frame $\left\{ X_1, \cdots, X_n \right\}$ and a local orthonormal vertical frame $\left\{ Z_1, \cdots, Z_m \right\}$ such that the following structure relations hold
\begin{align*}
& [X_i,X_j]=\sum_{k=1}^n \omega_{ij}^k X_k +\sum_{k=1}^m \gamma_{ij}^k Z_k,
\\
& [X_i,Z_k]=\sum_{j=1}^m \beta_{ik}^j Z_j,
\end{align*}
where $\omega_{ij}^k,  \gamma_{ij}^k,  \beta_{ik}^j $ are smooth functions such that
\[
 \beta_{ik}^j=- \beta_{ij}^k.
\]
Moreover, at $x_0$ we have
\[
 \omega_{ij}^k=0,  \beta_{ij}^k=0.
\]
\end{lemma}
We will also need the fact (see \cite[p. 918]{BaudoinKimWang2016}) that in this frame the Christoffel symbols of the Bott connection $\nabla$ are given by
\begin{align*}
& \nabla_{X_i} X_j =\frac{1}{2} \sum_{k=1}^n \left( \omega_{ij}^k +\omega_{ki}^j+\omega_{kj}^i\right)X_k,
\\
& \nabla_{Z_j} X_i =0,
\\
& \nabla_{X_i} Z_j=\sum_{k=1}^m \beta_{ij}^{k} Z_k.
\end{align*}
Thus, from the assumption that $D_\mathcal{H}=\nabla_\mathcal{H}$, we have
\begin{align*}
& D_{X_i} X_j =\frac{1}{2} \sum_{k=1}^n \left( \omega_{ij}^k +\omega_{ki}^j+\omega_{kj}^i\right)X_k, \\
& D_{X_i} Z_j=\sum_{k=1}^m \beta_{ij}^{k} Z_k.
\end{align*}
For $x_0 \in \M$ we let $\{ X_1\cdots, X_n, Z_1, \cdots, Z_m \}$  be a normal frame around $x_0$. If  $u \in \mathcal{O}_\mathcal{H} (\M)$ is a horizontal isometry, we can find an orthogonal matrix $\left\{ e_i^j \right\}_{i, j=1}^{n}$ such that $u(e_i)=\sum_{j=1}^n e_i^j X_j$, and  $u(f_i)=\sum_{j=1}^m f_i^j  Z_j$ for $f_i^j$, $i=1, ..., n$, $j=1, ..., m$. Let $\overline{X}_j$ be the vector field on $\mathcal{O}_\mathcal{H} (\M)$ defined by
\[
\overline{X}_j f (x,u)=\lim_{t \to 0} \frac{ f( e^{tX_j}(x), u)-f(x,u)}{t},
\]
where $e^{tX_j}(x)$ is the exponential map on $M$.
\begin{lemma}\label{l.4.2}
Let  $x_0 \in \M$ and  $(x,u) \in \mathcal{O}_\mathcal{H} (\M)$, then
\begin{align*}
& A_i(x,u)=\sum_{j=1}^n e_i^j \overline{X}_j
\\
& -\sum_{j,k,l,r=1}^n e_i^j e_r^l \langle D_{X_j} X_l, X_k\rangle \frac{\partial}{\partial e_r^k}-\sum_{j=1}^n\sum_{k,l,r=1}^m e_i^j f_r^l \langle D_{X_j} Z_l, Z_k\rangle \frac{\partial}{\partial f_r^k}.
\end{align*}
 In particular, at $x_0$ we have
\[
A_i(x_0,u)=\sum_{j=1}^n e_i^j \overline{X}_j.
\]
\end{lemma}

\begin{proof}
Let $u: \mathbb{R}^{n+m} \to T_x \M$ be a horizontal isometry and $x(t)$ be a smooth curve in $\M$ such that $x(0)=x$ and $x'(0)=u(e_i)$. We denote by $x^*(t)=(x(t),u(t))$ the $D$-lift to $\mathcal{O} (\M)$ of $x(t)$ and by $x'_1(t),\cdots,x'_n(t)$ the components of $x'(t)$ in the horizontal frame $X_1,\cdots,X_n$. Since $D$ is adapted to the foliation $\mathcal{F}$, the curve $x^*(t)$ takes its values in the horizontal frame bundle $\mathcal{O}_\mathcal{H} (\M)$. By definition of $A_i$, one has
\[
A_i=\sum_{j=1}^n x_j'(0) \overline{X}_j +\sum_{k,l=1}^n  u_{kl}'(0)\frac{\partial}{\partial e_k^l}+\sum_{k,l=1}^m  v_{kl}'(0)\frac{\partial}{\partial f_k^l},
\]
where $u_{kl}(t)=\langle u(t)(e_k), X_l \rangle$ and $v_{kl}(t)=\langle u(t)(f_k), Z_l \rangle$. Since $ u(t)(e_k)$ and $u(t)(f_k)$ are parallel along $x(t)$, one has
\[
D_{x'(t)} u(t)(e_k)=0, \quad D_{x'(t)} u(t)(f_k)=0.
\]
At $t=0$, this yields the expected result.
\end{proof}
In particular, Lemma \ref{l.4.2} implies the following statement.

\begin{proposition}
Let $\pi: \mathcal{O} (\M) \to \M$ be the bundle projection map. For a smooth $f:\M \to \mathbb{R}$, and $(x,u) \in \mathcal{O}_\mathcal{H} (\M) $,
\[
\left( \sum_{i=1}^n A_i^2 \right)(f\circ \pi)(x,u)=Lf \circ \pi (x,u).
\]
\end{proposition}

\begin{proof}
It is enough to prove this identity at $x_0$. Using the fact that at $x_0$ we have $\langle D_{X_j} X_l, X_k\rangle =\langle D_{X_j} Z_l, Z_k\rangle=0$, we see that
\[
\sum_{i=1}^n A_i^2 =\sum_{j=1}^n \overline{X}_j^2.
\]
The conclusion follows.
\end{proof}
As a straightforward corollary, we can introduce the \emph{horizontal} Brownian motion as follows.

\begin{corollary}\label{projec}
Let $(\Omega, (\mathcal{F}_s)_{s \geqslant 0}, \mathbb{P})$ be a filtered probability space that satisfies the usual conditions and let $\left\{ B_s \right\}_{s \geqslant 0}$ be an adapted $\mathbb{R}^{n}$-valued Brownian motion on that space. Let $\left\{ U_s \right\}_{s \geqslant 0}$ be a solution to the Stratonovich stochastic differential equation

\begin{equation}\label{e.HorizBM}
dU_s =\sum_{i=1}^n A_i (U_s) \circ dB^i_s= A_{U_s} \circ dB_s, \quad U_0 \in \mathcal{O}_\mathcal{H} (\M),
\end{equation}
then $W_s=\pi(U_s)$ is a \emph{horizontal Brownian motion} on $\M$, that is, a Markov process with the generator $\frac{1}{2} L$. Here we used Notation \ref{HVframe} and identified the $\mathbb{R}^{n}$-valued Brownian motion $\left\{ B_s \right\}_{s \geqslant 0}$  with an $\mathbb{R}^{n+m}$-valued process $\left( B_{s}, 0 \right)$.
\end{corollary}

\subsection{Horizontal semimartingales }

Let $(\Omega, (\mathcal{F}_{s})_{s \geqslant 0}, \mathbb{P})$ be a filtered probability space that satisfies the usual conditions.

\begin{definition}
An $\mathbb{R}^{n+m}$-valued $\mathcal{F}_{s}$-adapted continuous semimartingale $(W_{s})_{s \geqslant 0}$  is called \emph{horizontal} if for all $s \geqslant 0$
\[
\mathbb{P} \left(  W_s \in \mathbb{R}^n \times \left\{ 0 \right\} \right)=1.
\]
The space of horizontal semimartingales with $W_0=0$ will be denoted by $SW_{\mathcal{H}} (\mathbb{R}^{n+m})$.
\end{definition}

\begin{definition}
An $\mathbb{M}$-valued $\mathcal{F}_{s}$-adapted continuous semimartingale $\left\{  M_{s}  \right\}_{s \geqslant 0}$ is called \emph{horizontal} if for every vertical smooth  one-form $\theta$, and every $s \geqslant 0$ the Stratonovich stochastic line integral $\int_{M[0,s]} \theta =0$ almost surely.  The space of horizontal semimartingales such that $M_0=x_0$   will be denoted by $SW_{\mathcal{H}} (\M)$.
\end{definition}

\begin{remark}
We refer to \cite[Section 2.4, Definition 2.4.1]{HsuEltonBook} for the definition of Stratonovich stochastic line integrals.
\end{remark}
Then we have  the following result, whose proof is essentially identical to the proof of Lemma \ref{horizontal development map} and thus omitted for conciseness.

\begin{proposition}\label{Stochastic horizontal development}
As before $\pi$ is the bundle projection map $\mathcal{O}_\mathcal{H} (\M) \longrightarrow \M$.
\begin{enumerate}
\item Let $\left\{ W_s \right\}_{s\geqslant 0} \in SW_{\mathcal{H}} (\mathbb{R}^{n+m})$ and let $\left\{ U_s \right\}_{s \geqslant 0}$ be the solution to the Stratonovich stochastic differential equation
\[
dU_s =\sum_{i=1}^n A_i (U_s)  \circ dW^{i}_s=A_{U_s}\circ dW_s,\quad U_0 \in \mathcal{O}_\mathcal{H} (\M),
\]
then $M_s:=\pi(U_s)$ is a horizontal semimartingale on $\M$.
\item Let $\left\{ M_s \right\}_{s \geqslant 0} \in SW_{\mathcal{H}} (\M)$. Then there exists a unique $\left\{ W_s \right\}_{s\geqslant 0} \in SW_{\mathcal{H}} (\mathbb{R}^{n+m})$ such that if  $\left\{ U_s \right\}_{s \geqslant 0}$ is the solution to the Stratonovich stochastic differential equation
\[
dU_s =\sum_{i=1}^n A_i (U_s)  \circ dW^{i}_s=A_{U_s}\circ dW_s, \quad U_0\in \mathcal{O}_\mathcal{H} (\M),
\]
then $M_s=\pi(U_s)$ .
\end{enumerate}
\end{proposition}

Here we used Notation \ref{HVframe}, where we introduced how fundamental vector fields $A$ and $V$  on $\mathcal{O} (\M)$ in Notation \ref{HVframe} act on vectors in $\mathbb{R}^{n+m}$. Note that  $A$ acts on $\mathbb{R}^{n} \times \left\{ 0 \right\}$ in $\mathbb{R}^{n+m}$, and so we can apply it to $\omega^\mathcal{H}_s$.

Proposition \ref{Stochastic horizontal development} allows us to introduce the following notion.

\begin{definition}\label{d.HorizStochItoMaps} Suppose $\left\{ W_s \right\}_{s \geqslant 0}$ and $\left\{ M_s \right\}_{s \geqslant 0}$ are as in Proposition \ref{Stochastic horizontal development}. Then

(1) $\left\{ M_s \right\}_{s \geqslant 0}$ is called the \emph{stochastic  horizontal development} of $\left\{ W_s \right\}_{s \geqslant 0}$, and we denote $\phi_\mathcal{H} (W):=M$.

(2)  The path $\left\{ W_s \right\}_{s \geqslant 0}$ is called the \emph{stochastic  horizontal anti-development} of $\left\{ M_s \right\}_{s \geqslant 0}$, and we denote $\phi_{\mathcal{H}}^{-1} (M):=W$.

\end{definition}

As a consequence, one deduces that the horizontal Brownian motion constructed in Corollary \ref{projec}  is a horizontal semimartingale.

\begin{definition}\label{horizontal Ito map}
The \emph{horizontal It\^o map} (or \emph{horizontal stochastic development map}) is the following adapted map defined $\mu_\mathcal{H}$-a.s.
\begin{align*}
\phi_\Ho: & W_{0}\left( \mathbb{R}^n \right) \longrightarrow W_{x_0}\left( \M \right),
\\
& \omega^\mathcal{H} \longmapsto W
\end{align*}
\end{definition}
By using Proposition \ref{Stochastic horizontal development} and arguing as in \cite[p. 433]{Hsu1995b}, one can construct an adapted map $\phi_\Ho^{-1}: W_{x_0}\left( \M \right) \to W_{0}\left( \mathbb{R}^n \right)$ defined $\mu_W$-a.s. We will call $\phi_\Ho^{-1}$ the\emph{ stochastic horizontal  anti-development map}.

We also refer to \cite[Definition 2.5]{Driver1999b} for a discussion of the It\^o  map in the Riemannian setting and to the previous section for explicit constructions in our setting.
\begin{remark}
If one uses a Dirichlet form to construct the horizontal Brownian motion as in Section \ref{Dirichlet}, then it does not straightforward  to prove that one obtains a semimartingale. In particular, a standard approach such as the proof of \cite[Theorem 3.2.1]{HsuEltonBook}  does not readily extend to our setting.
\end{remark}

\subsection{Quasi-invariance of the horizontal Wiener measure}\label{s.QI}

In this section we prove quasi-invariance of the law of the horizontal Brownian motion with respect to variations generated by suitable \textit{tangent processes}.  Our argument follows  relatively closely the one by B.~Driver \cite{Driver1992b} and then E.~Hsu in \cite{Hsu1995b} (see also \cite{Cruzeiro1983b, CruzeiroMalliavin1996, EnchevStroock1995a}). More precisely, we will describe two types of variation of the  horizontal Brownian motion paths with respect to which the horizontal Wiener measure is quasi-invariant. The first one is largely inspired by Driver \cite[Theorem 7.28]{Driver2004a}. It is explicit, see Equation \eqref{StochVariationTangent} and readily yields the integration by parts formula in Section \ref{IBP from quasi}, but does not induce a flow. The second type of variation induces a flow and yields the  sub-Riemannian analogue of  \cite[Theorem 4.1]{Hsu1995b}.

\subsubsection{Framework}
We will use the same framework and notation as before. In particular, we still consider an arbitrary connection $D$ on $\M$  that satisfies the properties in Assumption \ref{property D}.
In addition we now introduce notation needed to establish quasi-invariance. We will mainly follow the presentation in \cite{Driver1999b, Driver2004a, Hsu1995b}.

We work in the probability space $(W_{0}\left( \mathbb{R}^n \right), \mathcal{B}, \mu_\mathcal{H})$, where $W_{0}\left( \mathbb{R}^n \right)$ is the space of continuous functions $\omega^\mathcal{H} : [0,1] \to \mathbb{R}^{n}$ such that $\omega^\mathcal{H}(0)=0$, $\mathcal{B}$ is the Borel $\sigma$-field on the path space $W_{0}\left( \mathbb{R}^n \right)$, and $\mu_\mathcal{H}$ is the Wiener measure. The coordinate process $(\omega^\mathcal{H}_s)_{0 \leqslant s \leqslant 1}$ is therefore a Brownian motion in $\mathbb{R}^{n}$. The usual completion of the natural filtration generated by $\left\{ \omega^\mathcal{H}_s \right\}_{0 \leqslant s \leqslant 1}$ will be denoted by $\mathcal{B}_s$.

We use the subscripts or superscripts $\mathcal{H}$, because, as before, $\mathbb{R}^{n}$ is identified with the subspace $\mathbb{R}^n\times \{ 0 \} \subset \mathbb{R}^{n+m}$. The $\mathbb{R}^{n+m}$-valued process $\left( \omega^\mathcal{H}_s,0 \right)$ will be referred to as a \emph{horizontal Brownian motion}.  The process  $\left\{ W_t \right\}_{0 \leqslant s \leqslant 1}$ constructed using Corollary \ref{projec} is the horizontal Brownian motion and the law $\mu_W$ of the horizontal Brownian motion on $\M$ will be referred to as the \emph{horizontal Wiener measure} on $\M$. Therefore, $\mu_W$ is a probability measure on the space $W_{x_0}\left( \M \right)$ of continuous paths $w:[0,1] \to \M$ with $w(0)=x_0$.

\begin{remark}
If the horizontal Laplacian can be written in  H\"ormander's form globally as in \ref{Hormander form}, then by \cite[Corollary 5.4]{ThalmaierCIRMLectures2016} the support of the horizontal Wiener measure $\mu_W$ is $W_{x_0}\left( \M \right)$ itself.
\end{remark}

\subsubsection{Tangent processes to the horizontal Brownian motion}

We now introduce the relevant class of tangent processes to the horizontal Brownian motion. To prove quasi-invariance, we consider the following class of tangent processes.

\begin{definition}
We define the \emph{horizontal Cameron-Martin space} denoted by $\mathcal{CM}_{\mathcal{H}} (\mathbb{R}^{n+m})$ as the space of  absolutely continuous  $\mathbb{R}^n$-valued (deterministic) functions $\left\{ h(s) \right\}_{0\leqslant s \leqslant 1}$ such that $h(0)=0$ and

\[
\int_0^1 \vert h^{\prime}(s) \vert_{\mathbb{R}^{n}}^2 ds  <\infty.
\]
\end{definition}

\begin{definition}\label{admissible2}
Suppose $\left\{ v(s)\right\}_{0 \leqslant s \leqslant 1}$ is a $\mathcal{B}_{s}$-adapted $\mathbb{R}^{n+m}$-valued continuous semimartingale such that

\begin{align}\label{e.3.8}
v(0)=0 \text{ and } \mathbb{E} \left( \int_0^1 \vert v(s) \vert_{\mathbb{R}^{n+m}}^2 ds \right)<\infty.
\end{align}
 The semimartingale $\left\{ v(s)\right\}_{0 \leqslant s \leqslant 1}$ will be called a \emph{tangent process} to the horizontal Brownian motion if the process
 \[
 v(s)-\int_0^s T_{U_r} (  A \circ d\omega^\mathcal{H}_r,  Av(r))
 \]
is a horizontal Cameron-Martin path, where $T$ denotes the torsion form of the Bott connection (not $D$). The space of  tangent processes to the horizontal Brownian motion will be denoted by $T W_\mathcal{H}(\mathbb{M})$.
\end{definition}

\begin{remark}
In Definition \ref{admissible2} we used the torsion $T$ of the Bott connection. Observe that since $T$ is a vertical tensor, a $\mathcal{B}_{s}$-adapted $\mathbb{R}^{n+m}$-valued continuous semimartingale $\left\{ v(s)\right\}_{0 \leqslant s \leqslant 1}$ satisfying \eqref{e.3.8} is in $T W_\mathcal{H}(\mathbb{M})$ if and only if
\begin{enumerate}
\item The horizontal part $v_\mathcal{H}$ is in  $\mathcal{CM}_{\mathcal{H}} (\mathbb{R}^{n+m})$;
\item The vertical part $v_\mathcal{V}$ is given by
\[
v_\mathcal{V} (s)=\int_0^s T_{U_r} (  A \circ d\omega^\mathcal{H}_r,  Av_\mathcal{H} (r)).
\]
\end{enumerate}
As a consequence, for any $h \in \mathcal{CM}_{\mathcal{H}} (\mathbb{R}^{n+m})$,
\begin{align}\label{ definition tau}
\tau^h(\omega^\mathcal{H})_s=h(s)+\int_0^s T_{U_r} (  A \circ d\omega^\mathcal{H}_r,  Ah (r))
\end{align}
is a tangent process to the horizontal Brownian motion.
\end{remark}
\begin{notation}
If $v \in T W_\mathcal{H}(\mathbb{M})$ is a tangent process, we  denote
\begin{align*}
 p_v(\omega^\mathcal{H})_s & := v(s)-\int_0^s T^D_{U_r} (  A \circ d\omega^\mathcal{H}_r,  Av(r)+Vv(r)) \\
 & -\int_0^s \left(  \int_0^r \Omega^D_{U_\tau}( A \circ  d \omega^\mathcal{H}_\tau, A v(\tau)+V v (\tau)) \right)  \circ  d \omega^\mathcal{H}_r,
\end{align*}
where $\Omega^D$ is the curvature form of the connection $D$.
\end{notation}
This definition comes from Equation \eqref{pullback1}, where $d\omega^\mathcal{H}$ is replaced by the Stratonovich differential $\circ d\omega^\mathcal{H}$. Since $D$ is a horizontal metric connection, the stochastic integral $ \int_0^s \Omega^D_{U_\tau}( A \circ  d \omega^\mathcal{H}_\tau, A v(\tau)+V v (\tau))$ restricts to $\mathbb{R}^n$ as a skew-symmetric endomorphism of $\mathbb{R}^n$. Also, from the proof of Theorem \ref{tangent deterministic} we have
\begin{align*}
& \int_0^s T^D_{U_r} (  A \circ d\omega^\mathcal{H}_r,  Av(r)+Vv(r))
\\
& = \int_0^s T_{U_r} (  A \circ d\omega^\mathcal{H}_r,  Av(r))- \int_0^s J_{V v(r)} ( A \circ d\omega^\mathcal{H}_r )_{U_r},
\end{align*}
where $J=D-\nabla$. As a consequence, $p_v(\omega^\mathcal{H})_s$ is actually a horizontal process, that is, it is $\mathbb{R}^n$-valued.

We can rewrite $p_v(\omega^\mathcal{H})_s$ by using  It\^o's integral, and we obtain
\begin{align*}
p_v(\omega^\mathcal{H})_s & =  v_\mathcal{H} (s)+ \frac{1}{2}\int_0^s \left( \mathfrak{Ric}^D_{\mathcal{H}}  \right)_{U_r} (Av(r)+Vv(r)) dr
\\
& + \int_0^s J_{V v(r)} ( A \circ d\omega^\mathcal{H}_r )_{U_r}
\\
&-\int_0^s \left(  \int_0^r \Omega^D_{U_\tau}( A \circ  d \omega^\mathcal{H}_\tau, A v(\tau)+V v (\tau)) \right)    d \omega^\mathcal{H}_r,
\end{align*}
where $\mathfrak{Ric}^D_{\mathcal{H}} $ is the horizontal Ricci curvature of the connection $D$.
We can further simplify this expression as follows.
\begin{align}\label{trace J}
 & J_{V v(s)} ( A \circ d\omega^\mathcal{H}_s )_{U_s} = \notag
 \\
 & \sum_{i=1}^n J_{V v(s)} ( A_i)_{U_s} \circ d\omega^i_s = \notag
 \\
 & J_{V v(s)} ( A\left( d\omega^\mathcal{H}_s \right) )_{U_s} +\frac{1}{2} \sum_{i=1}^n A_i J_{V v(s)} ( A_i)_{U_s}  ds \notag
 \\
 & -\frac{1}{2} \sum_{i=1}^n J_{T(A_i,Av_\mathcal{H} (s))} (A_i)_{U_s}ds
\end{align}
As a result, we see that

\begin{align*}
& p_v(\omega^\mathcal{H})_s  = v_\mathcal{H} (s)+\frac{1}{2} \sum_{i=1}^n \int_0^s A_i J_{V v(r)} ( A_i)_{U_r}  dr
\\
& -\frac{1}{2} \sum_{i=1}^n \int_0^s J_{T(A_i,Av_\mathcal{H} (r))} (A_i)_{U_r}dr
+ \frac{1}{2}\int_0^s \left( \mathfrak{Ric}^D_{\mathcal{H}}  \right)_{U_r} (Av(r)+Vv(r)) dr \\
&
+\int_0^s J_{V v(r)} ( A\left( d\omega^\mathcal{H}_r \right) )_{U_r} -\int_0^s \left(  \int_0^r \Omega^D_{U_\tau}( A \circ  d \omega^\mathcal{H}_\tau, A v(\tau)+V v (\tau)) \right)    d \omega^\mathcal{H}_r.
\end{align*}

More concisely, one can thus write
\begin{align}\label{component}
p_v(\omega^\mathcal{H})_s  =\int_0^s q_v(\omega^\mathcal{H})_r d\omega^\mathcal{H}_r + \int_0^s r_v(\omega^\mathcal{H})_r dr,
\end{align}
where $q_v$ is a $\mathfrak{so}(n)$-valued adapted process and $ r_v$ is an $\mathbb{R}^n$-valued adapted process such that  $\int_0^s |r_v(u)|^2_{\mathbb{R}^n} du < \infty$ a.s. The process $p_v$ is therefore an adapted vector field on $W_{0}\left( \mathbb{R}^n \right)$ in the sense of  \cite[Definition 3.2]{Driver1999b}.

\subsubsection{First type of variation}

We are now ready to construct the first relevant variation of the horizontal Brownian motion paths. The idea is to use the formula for the deterministic variation given by \ref{variation tangent} to infer a formula for a convenient stochastic variation.

\begin{notation}
For any $h \in \mathcal{CM}_{\mathcal{H}} (\mathbb{R}^{n+m})$ and any $t \in \R$, we denote by  $\rho_t^h: W_{0}\left( \mathbb{R}^{n} \right) \to W_{0}\left( \mathbb{R}^{n} \right)$ a map which is defined $\mu_\mathcal{H}$-a.s. as follows
\begin{align}\label{StochVariationTangent}
(\rho_t^h \omega_\mathcal{H})_s=\int_0^s e^{t q_{\tau_h (\omega^\mathcal{H})}(\omega^\mathcal{H})_u} d\omega^\mathcal{H}_u+t \int_0^s r_{\tau_h (\omega^\mathcal{H})}(\omega^\mathcal{H})_u du.
\end{align}
\end{notation}

\begin{remark}
As in the deterministic case, observe that $\rho^h$ is \textbf{not}
 the flow generated by $p_{\tau_h}$ on $W_{0}\left( \mathbb{R}^n \right)$. This variation  is similar to \cite[Theorem 7.28]{Driver2004a}. Let us however observe that $\mu_\mathcal{H}$ a.s., $\rho_0^h \omega^\mathcal{H}= \omega^\mathcal{H}$ and that from \eqref{component} one has
\[
\frac{d}{dt}\left|_{ t=0}  (\rho_t^h \omega^\mathcal{H})_s\right.=p_{\tau_h (\omega^\mathcal{H})}(\omega^\mathcal{H})_s.
\]
We also note that for every $t \in \R$, $e^{t q_{\tau_h }}$ is an $\mathbf{so}(n)$-valued process so that for every $t \in \R$  the semimartingale $\left\{ \int_0^s e^{t q_{\tau_h (\omega^\mathcal{H})}(\omega^\mathcal{H})_u} d\omega^\mathcal{H}_u\right\}_{s \in [0,1]}$ is a horizontal Brownian motion with respect to $\mu_\mathcal{H}$ .
\end{remark}

One has then the following analogue of \cite[Theorem 7.28]{Driver2004a} (see \cite{Driver1999b} for the details) which describes the differential of the horizontal stochastic development map and proves quasi-invariance of the horizontal Wiener measure for the variation described in \eqref{StochVariationTangent}.


\begin{thm}[Quasi-invariance I]\label{QI horizontal} Suppose $h \in \mathcal{CM}_{\mathcal{H}} (\mathbb{R}^{n+m})$.

(1) For every $t \in \R$ the law  of the semimartingale $\left\{ ( \rho_t^h \omega_\mathcal{H} )_s\right\}_{0 \leqslant s \leqslant 1}$  (under $\mu_\mathcal{H}$) is equivalent to $\mu_\mathcal{H}$, and the corresponding Radon-Nikodym density is given by
\begin{align*}
 & \frac{d(\rho_t^h)_{\ast} \mu_\mathcal{H}}{d \mu_\mathcal{H}}  (\omega^\mathcal{H})=
 \exp \left(t \int_0^{1}\left\langle   r_{\tau_h (\omega^\mathcal{H})}(\omega^\mathcal{H})_s, e^{t q_{\tau_h (\omega^\mathcal{H})}(\omega^\mathcal{H})_s} d\omega^\mathcal{H}_s \right\rangle
 \right.
 \\
 &
 \left.-\frac{t^2}{2} \int_0^{1} \vert r_{\tau_h (\omega^\mathcal{H})}(\omega^\mathcal{H})_s \vert_{\mathbb{R}^{n}}^2 ds\right).
\end{align*}

(2) For every $t \in \R$ the law  of the semimartingale $\left\{ \phi_\mathcal{H}(\rho_t^h \omega_\mathcal{H})_s\right\}_{0 \leqslant s \leqslant 1}$ (under $\mu_\mathcal{H}$) is equivalent to $\mu_W$  and the corresponding  Radon-Nikodym density is given by
\begin{align*}
 & \frac{d(\phi_\mathcal{H} \rho_t^h \phi_\mathcal{H}^{-1})_* \mu_W}{d \mu_W}  (w)
 =\exp \left(t \int_0^{1}\left\langle   r_{\tau_h (\omega^\mathcal{H})}(\omega^\mathcal{H})_s, e^{t q_{\tau_h (\omega^\mathcal{H})}(\omega^\mathcal{H})_s} d\omega^\mathcal{H}_s \right\rangle
 \right.
 \\
 &
 \left.
 -\frac{t^2}{2} \int_0^{1} \vert r_{\tau_h (\omega^\mathcal{H})}(\omega^\mathcal{H})_s \vert_{\mathbb{R}^{n}}^2 ds\right),
\end{align*}
where $\omega^\mathcal{H}=\phi_\mathcal{H}^{-1}(w)$.

(3) There exists a version of $\phi_\Ho ((\rho_t^h \omega_\mathcal{H}))_s$ which is continuous in $(s,t)$, differentiable in $t$, and such that
\[
\frac{d}{dt}\left|_{t=0} \phi_\Ho ((\rho_t^h \omega_\mathcal{H}))_s\right. =U_s\tau_h (\omega^\mathcal{H}) \hskip0.2in \mu_\mathcal{H}-a.s.
\]
\end{thm}

\begin{proof}
The first part  follows from  Girsanov's theorem in the form of \cite[Lemma 8.2]{Driver1992b}. The second part follows from Proposition \ref{v2jk} and is similar to  \cite[Theorem 7.28]{Driver2004a}.
\end{proof}

\subsubsection{Second type of variation}

 We now turn to the discussion of the stochastic flow generated by $p_{\tau^h}$.

\begin{notation}\label{n.semimartingale} For a fixed $h \in \mathcal{CM}_{\mathcal{H}} (\mathbb{R}^{n+m})$ we denote by $S\M_\mathcal{H} (h)$ the space of continuous and $\mathcal{B}_{s}$-adapted $\mathbb{R}^{n}$-valued semimartingales $\left\{ z_{s} \right\}_{0 \leqslant s \leqslant 1}$ that can be written as
\[
z_s= \int_0^s a_r dr +\int_0^s \sigma_r d\omega^\mathcal{H}_r, \quad 0 \leqslant s \leqslant 1,
\]
where $a$ is an $\mathbb{R}^n$-valued  $\mathcal{B}_{s}$-adapted process such that there exists a deterministic constant $C$
\begin{align}\label{estimate SMH}
\vert a_s \vert_{\mathbb{R}^{n}} \leqslant C( 1+\vert h^{\prime} (s)\vert_{\mathbb{R}^{n}}),
\end{align}
and where $\sigma$ is a $\mathcal{B}_{s}$-adapted process taking values in the space of isometries of $\mathbb{R}^n$.

\end{notation}


Observe that by  Girsanov's theorem in the form of \cite[Lemma 8.2]{Driver1992b}, the law of $z \in S\M_\mathcal{H} (h)$ is equivalent to the law $\mu_\mathcal{H}$ of the horizontal Brownian motion.
We are now in position to prove that $p_{\tau^h}$ generates a flow on the horizontal path space for which the horizontal Wiener measure on $\mathbb{R}^{n+m}$ is quasi-invariant. The following statement is similar to \cite[Theorem 3.1]{Hsu1995b}. The proof of that theorem relied on the Picard iteration to find a solution in a space of $\mathbb{R}^{n+m}$-valued continuous semimartingales equipped with a suitable norm. In our setting the proof is almost identical, so we omit it for conciseness.

\begin{thm}\label{flow euclidean}
For any $h \in \mathcal{CM}_{\mathcal{H}} (\mathbb{R}^{n+m})$ there exists a unique family of semimartingales $\{ \nu_t^h, t \in \mathbb{R}\}$ such that
\begin{itemize}
\item $\nu_t^h \in S\M_\mathcal{H} (h)$ for all $t\in \R$ and $\nu_0^h \omega^\mathcal{H}=\omega^\mathcal{H}$, $\mu_\mathcal{H}$ a.s.; hence the law of $\nu_t^h$ is equivalent to $\mu_\mathcal{H}$;
\item For $\mu_\mathcal{H}$-almost every $\omega^\mathcal{H}$, the function $t \longmapsto \nu_t^h\omega^\mathcal{H}$ is a $W_{0}\left( \mathbb{R}^n \right)$-valued continuous function;
\item $\mu_\mathcal{H}$-almost surely, $\nu_{t_1}^h \circ \nu_{t_2}^h (\omega^\mathcal{H}) =\nu_{t_1+t_2}^h(\omega^\mathcal{H}), \quad \text{for every} ~(t_1,t_2)\in\mathbb R \times \mathbb R $;
\item There exists a continuous version of $\{p_{\tau^h \nu_t^h } (\nu_t^h),t\in\R\}$ such that $\mu_\mathcal{H}$-almost surely, $\{\nu_t^h,t\in\R\}$ satisfies the equation
\begin{align}\label{flow equation}
\nu_t^h(\omega^\mathcal{H})=\omega^\mathcal{H}+\int_0^t p_{\tau^h (\nu_s^h (\omega^\mathcal{H})) } (\nu_s^h (\omega^\mathcal{H}) )ds.
\end{align}
\end{itemize}
\end{thm}

\begin{remark}\label{r.3.43}
In the previous theorem, the word unique is understood in the sense of \cite[Proposition 3.3]{Hsu1995b}, that is, in the space $S\M_\mathcal{H} (h)$.
\end{remark}

We are now  in position to prove quasi-invariance properties for the horizontal Wiener measure with respect to a suitable flow. The following statement is similar to Theorem 4.1 in \cite{Hsu1995b}. We recall that the horizontal stochastic development $\phi_\mathcal{H}$ and its inverse $\phi_\mathcal{H}^{-1}$ are defined in Definition \ref{horizontal Ito map}.

\begin{thm}[Quasi-invariance II]\label{main}
Let $h \in \mathcal{CM}_{\mathcal{H}} (\mathbb{R}^{n+m})$. The flow $\zeta^h_t = \phi_\mathcal{H} \circ \nu_t^h \circ \phi_\mathcal{H}^{-1} : W_{x_0}\left( \mathbb{M} \right) \to W_{x_0}\left( \mathbb{M} \right)$, $t \in \mathbb{R}$, is defined $\mu_W$-a.s. with the  generator $U\tau^h \phi_\mathcal{H}^{-1}$, and for every $t \in \R$ the distribution  of $ \zeta^h_t$ under $\mu_W$ is equivalent to $\mu_W$. More precisely, there exists a family of measurable maps
\[
\zeta_t^h:  W_{x_0}\left( \mathbb{M} \right) \to W_{x_0}\left( \mathbb{M} \right),\quad t\in \mathbb R,
\]
with the following properties.
\begin{itemize}
\item For every fixed $t\in \mathbb R$, the law $\mu_{\zeta_t^h}$ of $\zeta_t^h$ is equivalent to the horizontal Wiener measure $\mu_W$ and the Radon-Nikodym derivative is given by
\[
\frac{d\mu_{\zeta_t^h}}{d\mu_X}(w)=\frac{d\mu_{ \nu_t^h}}{d\mu_\mathcal{H}}(\phi_\mathcal{H}^{-1} w), \quad w \in W_{x_0}\left( \mathbb{M} \right).
\]
\item For $\mu_W$-almost every $w \in W_{x_0}\left( \mathbb{M} \right)$, the function $t\mapsto \zeta_t^h w$ is a $W_{x_0}\left( \mathbb{M} \right)$-valued continuous differentiable function;
\item For $\mu_W$-almost every $w \in W_{x_0}\left( \mathbb{M} \right)$, there is a continuous version of $t\mapsto U_t \tau^h(\phi_\mathcal{H}^{-1}\zeta_t^h w)$ such that $\zeta_v^t w$ satisfies the differential equation
\[
\frac{d\zeta^h_t w}{dt}=U_t \tau^h(\phi_\mathcal{H}^{-1}\zeta_t^h w);
\]
\item  $\mu_W$-almost surely,
\[
\zeta^h_{t_1}\circ \zeta_{t_2}^{h}=\zeta^h_{t_1+t_2},\quad \text{for all} ~~(t_1,t_2)\in\mathbb R \times \mathbb R.
\]
\end{itemize}

\end{thm}

\begin{proof}
The result  follows from Theorem \ref{flow euclidean}. For details, we refer to the proof of  \cite[Theorem 4.1]{Hsu1995b}.
\end{proof}

\subsubsection{Towards the integration by parts formulas}\label{IBP from quasi}

It is well known that a quasi-invariance result yields an integration by parts formula on the path space of the underlying diffusion, see B.~Driver \cite{Driver1992b} and then E.~Hsu \cite{Hsu1995b} (see also \cite{Cruzeiro1983b, CruzeiroMalliavin1996, EnchevStroock1995a}). Integration by parts formulas will be studied in  more detail in the second part of the paper, so we only briefly comment on the immediate corollary of Theorem \ref{QI horizontal}, which will be proved in another way (see Lemma \ref{IPP3}) and then extended to cylinder functions. It is obtained from Theorem \ref{QI horizontal} by taking the Bott connection $\nabla$ as the connection $D$,  and following the arguments of the proof in \cite[Theorem 7.32]{Driver2004a}.

\begin{lem}\label{IPP3}
Let $h \in \mathcal{CM}_{\mathcal{H}} (\mathbb{R}^{n+m})$, then for $f \in C^\infty(\M)$,
\begin{align*}
& \mathbb{E} \left( \left\langle df(W_1), U_1 \tau_h (\omega^\mathcal{H}) \right\rangle \right)
\\
& = \mathbb{E}  \left( f(W_1) \int_0^1 \left\langle h^{\prime}(s)+ \frac{1}{2}  (\mathfrak{Ric}_{\mathcal{H}})_{U_s}    h(s), d\omega_s^\mathcal{H} \right\rangle_{\mathbb{R}^{n}} \right),
\end{align*}
where $\mathbb{E}$ is the expectation with respect to $\mu_\mathcal{H}$ and $\mathfrak{Ric}_{\mathcal{H}}$ is the horizontal Ricci curvature of the Bott connection (viewed as an operator on $\mathbb{R}^n$).
\end{lem}

\subsubsection{The case of a Riemannian submersion: examples}\label{Example QI}


To finish this part of the paper, we  discuss the case when  the foliation on $\M$ comes from a totally geodesic submersion $\pi : (\M,g) \to (\B,j)$ as described in Example \ref{ex.RiemSubmersion}. This should allow the reader to relate our quasi-invariance result  to the Riemannian result by  B.~Driver in \cite{Driver1992b}.

In the submersion case, the notion of horizontal lift of curves plays an important role.

\begin{definition}
Let $ \overline{\gamma}: [0, \infty) \to \mathbb{B}$ be a $C^1$-curve. Let $x \in \mathbb{M}$, such that $\pi(x)=\gamma(0)$. Then, there exists a unique $C^1$-horizontal curve $\gamma: [0, \infty) \to \mathbb M$ such that $\gamma (0)=x$ and $ \pi (\gamma(t))=\overline{ \gamma}(t)$. The curve $\gamma$ is called the horizontal lift of $\overline{\gamma}$ at $x$.
\end{definition}
The notion of horizontal lift may be extended to Brownian motion paths on $\mathbb{B}$ by using stochastic calculus. The argument is similar to the case of the  stochastic lift of the Brownian motion of a Riemannian manifold to the orthonormal frame bundle, see for instance \cite[Theorem 3.2]{Driver1992b}).

The submersion has totally geodesic fibers, therefore $\pi$ is harmonic and the projected process
\[
W_t^\B =\pi (W_t)
\]
is, under $\mu_\mathcal{H}$, a Riemannian Brownian motion on $\B$ started at $\pi(x_0)$. The submersion $\pi$ induces a map $W_{x_0}\left( \mathbb{M} \right) \to W_{\pi(x_0)}\left( \B \right)$ that we still denote by $\pi$. Let now $h$ be a Cameron-Martin path in $\mathbb{R}^n$ and consider the flow $\zeta^h_t  : W_{x_0}\left( \mathbb{M} \right) \to W_{x_0}\left( \mathbb{M} \right)$, $t \in \mathbb{R}$, which is defined $\mu_W$-a.s. according to Theorem \ref{main}. By using the   horizontal stochastic lift $W_{\pi(x_0)}\left( \B \right) \to W_x\left( \mathbb{M} \right)$, one can construct a flow $\tilde{\zeta}^h_t : W_{\pi(x_0)}\left( \B \right) \to W_{\pi(x_0)}\left( \B \right)$, $t \in \mathbb{R}$ which is unique  $\mu_{W^\B}$-a.s. as mentioned in Remark \ref{r.3.43}. Then we have the following commutative diagram


\begin{equation}\label{lk}
\begin{tikzcd}
W_{x_0}\left( \mathbb{M} \right) \arrow{r}{\zeta^h_t} \arrow[swap]{d}{\pi} & W_{x_0}\left( \mathbb{M} \right) \arrow{d}{\pi} \\
W_{\pi(x_0)}\left( \B \right) \arrow{r}{\tilde{\zeta}^h_t} & W_{\pi(x_0)}\left( \B \right)
\end{tikzcd}
\end{equation}

By Theorem \ref{main},  the law of $W^\B$ is quasi-invariant under the flow  $\tilde{\zeta}^h_t$. Note that the connection $D$ projects down to the Levi-Civita connection on $\B$, therefore the flow $\tilde{\zeta}^h_t$ provides a version of the flow considered by E.~Hsu in \cite[Theorem 4.1]{Hsu1995b}. Thus we recover Driver's quasi-invariance result \cite{Driver1992b} on the manifold $\B$. Further details on this example will be given in Section \ref{ss.RiemSubmersions}, where the generator $\tilde{\zeta}^h_t$ will be computed explicitly.

It may be useful to illustrate the diagram \eqref{lk}  in a very simple situation. Recall that the Heisenberg group is the set
\[
\mathbb{H}^{2n+1}=\left\{ (x, y, z), x \in \mathbb{R}^n,  y \in \mathbb{R}^n, z\in\mathbb{R} \right\}
\]
endowed with the group law
\[
(x_1,y_1,z_1) \star (x_2,y_2,z_2):=(x_1+x_2,y_1+y_2,z_1+z_2+\langle x_1,y_2 \rangle_{\R^n} -\langle x_2,y_1 \rangle_{\R^n}).
\]
The vector fields
\begin{align*}
& X_i=\frac{\partial}{\partial x_i} -y_i \frac{\partial}{\partial z},  1 \leqslant i \leqslant n,
\\
& Y_i=\frac{\partial}{\partial y_i} +x_i \frac{\partial}{\partial z},  1 \leqslant i \leqslant n,
\\
& Z=\frac{\partial}{\partial z}
\end{align*}
form a basis for the space of left-invariant vector fields on $\mathbb{H}^{2n+1}$. We choose a left-invariant Riemannian metric on $\mathbb{H}^{2n+1}$ in such a way that $\left\{ X_{1}, ..., X_{n}, Y_{1}, ..., Y_{n}, Z \right\}$ are orthonormal with respect to this metric. Note that these vector fields satisfy the following commutation relations
\[
[X_i,Y_j]=2\delta_{ij} Z, \hskip0.2in [X_i,Z]=[Y_i,Z]=0, \hskip0.2in i=1, ..., n.
\]
Then, the projection map

\begin{align*}
\begin{array}{llll}
\pi:  & \mathbb{H}^{2n+1} & \longrightarrow & \mathbb{R}^{2n}
\\
      & (x,y,z)  & \longmapsto & (x,y)
\end{array}
\end{align*}
is a Riemannian submersion with totally geodesic fibers. In that example, the Bott connection is trivial: $\nabla X_i=\nabla Y_j=\nabla Z=0$ and its torsion is given by
\[
T(X_i,Y_j)=-2\delta_{ij} Z, \quad T(X_i,Z)=T(Y_i,Z)=0.
\]

 Let  now $W_{0}\left( \mathbb{R}^{2n} \right)$ be the Wiener space of continuous functions $[0, 1]\to \mathbb{R}^{2n}$ starting at $0$. We denote by $(B_t,\beta_t)_{0 \leqslant t \leqslant 1}$ the coordinate maps on $W_{0}\left( \mathbb{R}^{2n} \right)$ and by $\mu_\mathcal{H}$ the Wiener measure on $W_{0}\left( \mathbb{R}^{2n} \right)$, so that $(B_t, \beta_t)_{0 \leqslant t \leqslant 1 }$ is a $2n$-dimensional Brownian motion under $\mu_\mathcal{H}$. By using the submersion $\pi$, the Brownian motion $(B_t,\beta_t)_{0 \leqslant t \leqslant 1}$ can be horizontally lifted to the horizontal Brownian motion on $\mathbb{H}^{2n+1}$ which is given explicitly by
 \[
 W_t=\left(B_t,\beta_t, \sum_{i=1}^n \int_0^t B^i_t d\beta^i_t -\beta^i_t dB^i_t \right).
 \]

 Let $h=(h_1,h_2)$ be a Cameron-Martin path in $\mathbb{R}^{2n}$ and consider the Cameron-Martin flow $\tilde{\zeta}^h_t  : W_{0}\left( \mathbb{R}^{2n} \right) \to W_{0}\left( \mathbb{R}^{2n} \right)$, $t \in \mathbb{R}$, explicitly given by
 \[
 \tilde{\zeta}^h_t (B,\beta)= (B,\beta) +t h.
 \]
 One has then a commutative diagram

\begin{equation}\label{diagram 2}
\begin{tikzcd}
W_{0}\left( \mathbb{H}^{2n+1} \right) \arrow{r}{\zeta^h_t} \arrow[swap]{d}{\pi} & W_{0}\left( \mathbb{H}^{2n+1} \right) \arrow{d}{\pi}
\\
W_{0}\left( \mathbb{R}^{2n} \right) \arrow{r}{\tilde{\zeta}^h_t} & W_{0}\left( \mathbb{R}^{2n} \right)
\end{tikzcd}
\end{equation}
where  $\zeta^h_t$ is the flow on $W_{0} \left( \mathbb{H}^{2n+1} \right)$ defined $\mu_\mathcal{H}$-a.~s.  by
\begin{align*}
  & \zeta^h_t \left(W \right)= \left(B+th_1,\beta+th_2, \right.
  \\
& \left. \sum_{i=1}^n \int_0^\cdot (B^i_u+th_1^i (u)) d(\beta^i_u +th_2^i(u))-(\beta^i_u +th_2^i(u)) d(B^i_u+th^i_1(u)) \right).
\end{align*}
One can compute the generator of this flow as
\begin{align*}
 & \frac{d}{dt}\left|_{t=0} \zeta^h_t \left( W \right)\right.
 \\
  = &\left(h_1,h_2, \sum_{i=1}^n \int_0^\cdot h_1^i (u) d\beta^i_u -h_2^i(u) dB^i_u + \sum_{i=1}^n \int_0^\cdot B_u^i dh^i_2 (u) -\beta^i_u dh^i_1(u) \right)
  \\
  =& \left(h_1,h_2,   \sum_{i=1}^n  B^i h^i_2 -\beta^i h^i_1+2\sum_{i=1}^n \int_0^\cdot h_1^i (u) d\beta^i_u -h_2^i(u) dB^i_u  \right)
  \\
  =&\sum_{i=1}^{n} h_1^i X_i(W) +\sum_{i=1}^n   h_2^iY_i(W) +2\left( \sum_{i=1}^n \int_0^\cdot h_1^i (u) d\beta^i_u -h_2^i(u) dB^i_u \right) Z (W)
\end{align*}

As expected, we can interpret this generator in terms of the Bott connection as a straightforward computation shows that
\begin{align*}
 &  \int_0^s T \left( \sum_{i=1}^n X_i \circ dB^i_u +\sum_{i=1}^n Y_i  \circ d \beta^i_u , \sum_{i=1}^n h^i _1 (u) X_i + \sum_{i=1}^n h^i_2(u) Y_i \right)\\
 =& \left( 2 \sum_{i=1}^n \int_0^s h^i _1 (u) d \beta^i_u -2 \sum_{i=1}^n  \int_0^s h^i_2(u) dB^i_u  \right) Z(W)
\end{align*}
Therefore, we showed that
\begin{align*}
& \frac{d}{dt}\left|_{t=0} \zeta^h_t \left( W \right) \right.
  \\
=&\sum_{i=1} h_1^i X_i(W) +\sum_{i=1}^n   h_2^iY_i(W)
\\
&+ \int_0^\cdot T \left( \sum_{i=1}^n X_i \circ dB^i_u +\sum_{i=1}^n Y_i  \circ d \beta^i_u , \sum_{i=1}^n h^i _1 (u) X_i + \sum_{i=1}^n h^i_2(u) Y_i \right)
\end{align*}
This is exactly Equation \eqref{ definition tau} written in the parallel frame $\{ X_i,  Y_j, Z \}_{i=1}^{n}$.

\section{Integration by parts formulas}\label{p.2}


The goal of this part of the paper is to establish several types of integration by parts formulas for the horizontal Brownian motion. This part relies on very different techniques than the ones used in the first part and therefore we  need to introduce more notation.  Though we will consider  the horizontal Brownian motion constructed from the frame bundle,  in this part of the paper   we will rely on the stochastic parallel transport rather than the stochastic lift to the frame bundle (although these are of course equivalent). Also, instead of working with general connections denoted by $D$ in Section \ref{p.1}, we now consider connections satisfying Assumption \ref{property D} and with the additional property that the torsion satisfies B.~Driver's anti-symmetry condition \cite[Definition 8.1]{Driver1992b}. Throughout this part, we will work with the following probability space.
\begin{notation}
By $(\Omega, \mathcal{B}, \mu_\mathcal{H})$ we denote the probability space, where  $\Omega:=W_{0}\left( \mathbb{R}^n \right)$ is the space of continuous functions $\omega^\mathcal{H} : [0,1] \to \mathbb{R}^{n}$ such that $\omega^\mathcal{H}(0)=0$, $\mathcal{B}$ is the Borel $\sigma$-field on $W_{0}\left( \mathbb{R}^n \right)$, and $\mu_\mathcal{H}$ is the Wiener measure on $\Omega$. Then the coordinate process $\left\{ \omega^\mathcal{H}_s \right\}_{0 \leqslant s \leqslant 1}$ is a Brownian motion with values in $\mathbb{R}^{n}$. The usual completion of the natural filtration generated by $\left\{ \omega^\mathcal{H}_s \right\}_{0 \leqslant s \leqslant 1}$  will be denoted by $\mathcal{F}_s$.
\end{notation}
Recall that for $x \in \M$ the horizontal Brownian motion on $\M$ started at $x$ is defined as $W_s=\pi(U_s)$, where $U_{s}$ is a solution to the Stratonovich stochastic differential equation
\eqref{e.HorizBM} with $U_0=u_0 \in \mathcal{O}_\mathcal{H} (\M)$ such that $\pi(u_0)=x$.

\subsection{Horizontal Weitzenb\"ock type formulas}

We start by introducing a family of connections that will be of interest to us later, and we review some known results on the Weitzenb\"ock formulas proved previously in \cite{BaudoinKimWang2016}.

\subsubsection{Generalized Levi-Civita connections and adjoint connections}\label{W:6}
In Section \ref{s.WeitzFormula} we aim at studying Weitzenb\"ock-type identities for the horizontal Laplacian, and for this we need to introduce a new class of connections. The main reason why we use these connections is that we can not make use of the Bott connection since the adjoint connection to the Bott connection is not metric. We refer to \cite{BaudoinKimWang2016, GrongThalmaier2016a, GrongThalmaier2016b, Elworthy2017} and especially the books \cite{ElworthyLeJanLiBook1999, ElworthyLeJanLiBook2010} for a discussion on Weitzenb\"ock-type identities and adjoint connections. Instead we make use of the family of connections first  introduced  in \cite{Baudoin2017b} and only keep the Bott connection as a reference connection.

This family of connections is constructed from a natural variation of the metric that we recall now. The Riemannian metric $g$ can be split using horizontal and vertical subbundles described in Section \ref{ss.HorizVertSubbundles}

\begin{equation}\label{e.1}
g=g_\mathcal{H} \oplus g_{\mathcal{V}}.
\end{equation}
Using the splitting of the Riemannian metric $g$ in \eqref{e.1} we can introduce the following one-parameter family of Riemannian metrics

\[
g_{\varepsilon}=g_\mathcal{H} \oplus  \frac{1}{\varepsilon }g_{\mathcal{V}}, \quad \varepsilon >0.
\]
One can check that for every $\varepsilon >0$, $\nabla g_\varepsilon =0$ where $\nabla$ is the Bott connection. The metric $g_\varepsilon$ then induces a metric on the cotangent bundle which we still denote by  $g_\varepsilon$, and therefore
\[
\| \eta \|^2_{\varepsilon} =\| \eta \|_\mathcal{H}^2+\varepsilon \| \eta \|_\mathcal{V}^2, \text{ for every } \eta \in T^*_x \M.
\]

For each $Z \in \Gamma^\infty(\mathcal{V})$ there is a unique skew-symmetric endomorphism  $J_Z: \mathcal{H}_x \to \mathcal{H}_x$, $x \in \M$ such that for all horizontal vector fields $X, Y \in \mathcal{H}_{x}$
\begin{align}\label{Jmap}
g_\mathcal{H} (J_Z (X), Y)_{x}= g_\mathcal{V} (Z, T(X, Y))_{x},
\end{align}
where $T$ is the torsion tensor of $\nabla$. We then extend $J_{Z}$ to be $0$ on  $\mathcal{V}_x$. Also, to ensure \eqref{Jmap} holds also for $Z\in \Gamma^\infty (\mathcal{H})$, taking into account \eqref{e.1.2} we set $J_Z \equiv 0$.

Following \cite{Baudoin2017b} we  introduce  the following family of connections
\[
\nabla^\varepsilon_X Y= \nabla_X Y -T(X,Y) +\frac{1}{\varepsilon} J_Y X, \quad X,Y\in \Gamma^\infty(\M).
\]
It is easy to check that $\nabla^\varepsilon g_\varepsilon =0$  and the torsion of $\nabla^\varepsilon$ is given by
\[
T^{\varepsilon}(X,Y)=-T(X,Y)+\frac{1}{\varepsilon} J_YX-\frac{1}{\varepsilon} J_XY, \quad X,Y \in \Gamma^\infty(\M).
\]
The adjoint connection to $\nabla^\varepsilon$ as described by B.~Driver in \cite{Driver1992b}, see also \cite[Section 1.3]{ElworthyLeJanLiBook1999} for a discussion about adjoint connections,   is then given by
\begin{equation}\label{e.adjointconn}
\widehat{\nabla}^\varepsilon_X Y:=\nabla^\varepsilon_X Y -T^\varepsilon (X,Y) =\nabla_X Y +\frac{1}{\varepsilon} J_X Y,
\end{equation}
thus $\widehat{\nabla}^\varepsilon$ is also a metric connection. Moreover, it preserves the horizontal and vertical bundles.
\begin{remark}\label{compatibility F}
Note that the connection $\widehat{\nabla}^\varepsilon$ therefore satisfies  Assumption \ref{property D} for every $\ve >0$.
\end{remark}

For later use, we record that the torsion of $\widehat{\nabla}^\varepsilon$ is
\begin{equation}\label{e.4.3}
\widehat{T}^\varepsilon (X,Y)=-T^{\varepsilon}(X,Y)= T(X,Y)-\frac{1}{\varepsilon} J_Y X+\frac{1}{\varepsilon} J_X Y.
\end{equation}
The Riemannian curvature tensor of $\widehat{\nabla}^\varepsilon$  can be computed explicitly  in terms of the Riemannian curvature tensor $R$ of the Bott connection $\nabla$ and it is given by the following lemma.
\begin{lem}\label{curvature adjoint}
For $X, Y, Z \in \Gamma^\infty(\M)$
\begin{align*}
\widehat{R}^\varepsilon (X,Y)Z= & R(X,Y)Z +\frac{1}{\varepsilon} J_{T(X,Y)} Z +\frac{1}{\varepsilon^2} (J_X J_Y -J_Y J_X)Z  +
\\
& \frac{1}{\varepsilon} (\nabla_X J)_Y  Z -  \frac{1}{\varepsilon} (\nabla_Y J)_X  Z,
\end{align*}
where $R$ is the curvature tensor of the Bott connection.
\end{lem}

\begin{proof}
\begin{align*}
& \widehat R^{\varepsilon}(X,Y)Z=\widehat \nabla^{\varepsilon}_X\widehat \nabla_Y^{\varepsilon}Z-\widehat \nabla^{\varepsilon}_Y\widehat \nabla^{\varepsilon}_XZ-\widehat \nabla^{\varepsilon}_{[X,Y]}Z\\
=&(\nabla_X\nabla_Y+\frac{1}{\varepsilon}(\nabla_X J)_Y +\frac{1}{\varepsilon} J_X\nabla_Y+\frac{1}{\varepsilon} J_Y\nabla_X+\frac{1}{\varepsilon} J_{\nabla_X Y}+\frac{1}{\varepsilon^2} J_X J_Y  )Z\\
&-(\nabla_Y\nabla_X+\frac{1}{\varepsilon}(\nabla_Y J)_X+\frac{1}{\varepsilon} J_Y\nabla_X+\frac{1}{\varepsilon} J_{\nabla_Y X}+\frac{1}{\varepsilon} J_X\nabla_Y+\frac{1}{\varepsilon^2}J_Y J_X)Z\\
&-\nabla_{[X,Y]}Z-\frac{1}{\varepsilon} J_{[X,Y]}Z\\
=&R(X,Y)Z+\frac{1}{\varepsilon^2}( J_X J_Y- J_Y J_X)Z+
\\
& \frac{1}{\varepsilon}(\nabla_X  J)_YZ-\frac{1}{\varepsilon}(\nabla_Y J)_XZ +\frac{1}{\varepsilon} J_{T(X,Y)} Z.
\end{align*}
\end{proof}

We define  the horizontal Ricci curvature $\mathfrak{Ric}_{\mathcal{H}}$ for the Bott connection as the fiberwise symmetric linear map on one-forms such that for all smooth functions $f, g$ on $\M$
\[
\langle  \mathfrak{Ric}_{\mathcal{H}} (df), dg \rangle=\mathbf{Ric} \left(\nabla_\mathcal{H} f, \nabla_\mathcal{H} g\right)=\mathbf{Ric}_\mathcal{H} ( \nabla f, \nabla g),
\]
where $\mathbf{Ric}$ is the Ricci curvature of the Bott connection $\nabla$ and $\mathbf{Ric}_\mathcal{H}$ is its horizontal Ricci curvature (horizontal trace of the full curvature tensor $R$  of the Bott connection). The fact that $\mathbf{Ric}_\mathcal{H}$ is symmetric follows from  \cite[Lemma 4.2 ]{Hladky2012}.

\subsubsection{Weitzenb\"ock formulas}\label{s.WeitzFormula}

A key ingredient in studying the horizontal Brownian motion  is the Weitzenb\"{o}ck formula that has been proven in \cite{Baudoin2017b, BaudoinKimWang2016}. We recall here this formula. If $Z_1, \dots, Z_m$ is a local vertical frame, then the $(1,1)$ tensor

\[
\mathbf{J}^2 := \sum_{\ee=1}^m J_{Z_\ee}J_{Z_\ee}
\]
does not depend on the choice of the frame and may be defined globally.

\begin{example}[Example \ref{ex.Kcontact} revisited]  If $\M$ is a K-contact manifold equipped with the Reeb foliation, then, by taking $Z$ to be the Reeb vector field, one gets $\mathbf{J}^2=J_Z^2=-\mathbf{Id}_{\mathcal{H}}$.
\end{example}

The \emph{horizontal divergence} of the torsion $T$ is the $(1,1)$ tensor  which in a local horizontal frame $X_1, \dots, X_n$  is defined by
\begin{equation}\label{e.HorizDivergence}
\delta_\mathcal{H} T (X):= -\sum_{j=1}^n(\nabla_{X_j} T) (X_j,X).
\end{equation}


By using the duality between the tangent and cotangent bundles with respect to the metric $g$, we can identify the $(1,1)$ tensors $\mathbf{J}^2$ and $\delta_\mathcal{H} T$ with linear maps on the cotangent bundle $T^* \M$.

Namely, let $\sharp: T^{\ast}\M \rightarrow T\M$ be the standard musical (raising an index) isomorphism which is defined as the unique vector $\omega^{\sharp}$ such that for any $x \in \M$

\[
g\left( \omega^{\sharp}, X \right)_{x}=\omega \left( X \right) \text{ for all } X \in T_{x}\M,
\]
while in local coordinates the isomorphism $\sharp$ can be written as follows

\[
\omega=\sum_{i=1}^{n+m}\omega_{i}dx^{i} \longmapsto \omega^{\sharp}= \sum_{j=1}^{n+m}\omega^{j}\partial_{j}=\sum_{j=1}^{n+m}\sum_{i=1}^{n+m}g^{ij}\omega_{i}\partial_{j}.
\]
The inverse of this isomorphism is the (lowering an index) isomorphism $\flat: T\M \rightarrow T^{\ast}\M$ defined by
\[
X^{\flat}=g\left( X, \cdot \right)_{x}, X \in T_{x}\M
\]
and in local coordinates

\[
X=\sum_{i=1}^{n+m}X^{i}\partial_{i} \longmapsto X^{\flat}=\sum_{i=1}^{n+m} X_{i}dx^{i}=\sum_{i=1}^{n+m} \sum_{j=1}^{n+m}g_{ij}X^{j}dx^{i}.
\]

If $\eta$ is a one-form, we define the horizontal gradient in a local adapted frame of $\eta$ as the $(0,2)$ tensor
\[
\nabla_\mathcal{H} \eta =\sum_{i=1}^n \nabla_{X_i} \eta \otimes \theta_i,
\]
where $\theta_i, i=1,\dots, n$ is the dual to $X_i$.

Finally, for $\varepsilon >0 $, we consider the following operator  which is defined on one-forms by
\begin{align}\label{def}
\square_\varepsilon:=\sum_{i=1}^n (\nabla_{X_i} -\mathfrak{T}^\varepsilon_{X_i})^2 - ( \nabla_{\nabla_{X_i} X_i}-  \mathfrak{T}^\varepsilon_{\nabla_{X_i} X_i})-\frac{1}{ \varepsilon}\mathbf{J}^2+\frac{1}{\varepsilon} \delta_\mathcal{H} T- \mathfrak{Ric}_{\mathcal{H}},
\end{align}
where $\mathfrak{T}^\varepsilon$ is the $(1,1)$ tensor defined by

\[
\mathfrak{T}^\varepsilon_X Y=-T(X,Y) +\frac{1}{\varepsilon} J_Y X, \quad X,Y\in \Gamma^\infty(\M).
\]
Similarly as before, we will use the notation
\[
\mathfrak{T}^\varepsilon_\mathcal{H} \eta :=\sum_{i=1}^n \mathfrak{T}^\varepsilon_{X_i} \eta  \otimes \theta_i.
\]
The expression in \eqref{def} does not depend on the choice of the local horizontal frame and thus $\square_\varepsilon$ may be globally defined. Formally, we have
\begin{align}\label{def2}
\square_\varepsilon=-(\nabla_\mathcal{H} -\mathfrak{T}_\mathcal{H}^\varepsilon)^* (\nabla_\mathcal{H} -\mathfrak{T}_\mathcal{H}^\varepsilon)-\frac{1}{ \varepsilon}\mathbf{J}^2+\frac{1}{\varepsilon} \delta_\mathcal{H} T- \mathfrak{Ric}_{\mathcal{H}},
\end{align}
where the adjoint is understood with respect to the $L^2\left( \M, g_\varepsilon, \mu \right)$ inner product on sections, i.e. $\int_\M \langle\cdot,\cdot\rangle_{\varepsilon}d\mu$ (see \cite[Lemma 5.3]{BaudoinEMS2014} for more detail). The main result in \cite{BaudoinKimWang2016} is the following. Here the Laplacian $L$ is defined by Equation \eqref{Hormander form} in Section \ref{Dirichlet}
\begin{thm}[Lemma 3.3, Theorem 3.1 in  \cite{BaudoinKimWang2016}]
\label{Weitzenbock}
Let $f \in C^\infty(\M)$,  $x \in \M$ and  $\varepsilon>0$, then
\begin{equation}\label{pol}
dLf (x)=\square_\varepsilon df(x),
\end{equation}
where $L$ is defined by Equation \eqref{Hormander form}.
\end{thm}

\begin{rem}
Using \cite[Lemma 3.4]{BaudoinKimWang2016}, we see that for $\varepsilon_1,\varepsilon_2 >0$, the operator $\square_{\varepsilon_1}-\square_{\varepsilon_2}  $ vanishes on exact one-forms. It is therefore no surprise that the left hand side of \eqref{pol} does not depend of $\varepsilon$.
\end{rem}

To conclude this section we remark, and this is not a coincidence,  that the potential term in the Weitzenb\"ock identity can be identified with the horizontal  Ricci curvature of the adjoint connection $\widehat{\nabla}^\varepsilon$.

\begin{lemma}\label{ricci adjoint}
The horizontal  Ricci curvature of the adjoint connection $\widehat{\nabla}^\varepsilon$ is given by
\[
\Hat{\mathfrak{Ric}}_{\mathcal H}^{\varepsilon}=\mathfrak{Ric}_{\mathcal H}-\frac{1}{\varepsilon}\delta^*_{\mathcal H}T+\frac{1}{\varepsilon}\mathbf J^2,
\]
where $\delta^*_{\mathcal H}T$ denotes the adjoint of $\delta_{\mathcal H}T$ with respect to the metric $g$.
\end{lemma}

\begin{proof}
Let $X, Y \in \Gamma^\infty( T \M )$ and $X_1,\cdots,X_n$ be a local horizontal orthonormal frame. By the definition of the horizontal Ricci curvature  and Lemma \ref{curvature adjoint} we have
\begin{align*}
 & \Hat{\mathbf{Ric}}_\mathcal{H} ^{\varepsilon}(X,Y)\\
 =& \sum_{i=1}^ng_{\mathcal H}(\Hat{R}^{\varepsilon}(X_i,X)Y,X_i)\\
=&\sum_{i=1}^ng_{\mathcal H}(R(X_i,X)Y,X_i)+\sum_{i=1}^ng_{\mathcal H}\left(\frac{1}{\varepsilon} J_{T(X_i,X)}Y,X_i \right)\\
&+\sum_{i=1}^ng_{\mathcal H}\left(\frac{1}{\varepsilon}(\nabla_{X_i}  J)_{X}Y-\frac{1}{\varepsilon}(\nabla_{X} J)_{X_i} Y,X_i\right).
\end{align*}
For the first term, we have
\[
\sum_{i=1}^ng_{\mathcal H}(R(X_i,X)Y,X_i)=  \mathbf{Ric}_\mathcal{H} (X,Y).
\]
For the second term, we easily see that
\begin{align*}
\sum_{i=1}^ng_{\mathcal H}\left( J_{T(X_i,X)}Y,X_i \right)
 & =-\sum_{i=1}^ng_{\mathcal V}\left( T(X,X_i),T(Y,X_i)  \right) \\
 &=g_\mathcal{H} ( \mathbf{J}^2 X, Y).
\end{align*}
For the third term, we first observe that $g_\mathcal{H} ( (\nabla_X  J)_{X_i}Y,X_i) =0$. Then, we have
\begin{align*}
\sum_{i=1}^ng_{\mathcal H}\left( (\nabla_{X_i} J)_X Y,X_i\right)& =-\sum_{i=1}^ng_{\mathcal H}\left( (\nabla_{X_i} J)_X X_i,Y\right) \\
& =-\sum_{i=1}^ng_{\mathcal V}\left( (\nabla_{X_i} T)(X_i, Y),X\right) \\
 &=g_{\mathcal V}\left( \delta_\mathcal{H} T (Y),X \right).
\end{align*}
\end{proof}

\subsection{Integration by parts formula on the horizontal path space}\label{s.IbyP}

We fix $\varepsilon >0$ throughout the section. Our goal in this section is to prove integration by parts formulas on the path space of the horizontal Brownian motion. Some of the integration by parts formulas for the damped Malliavin derivative have been already announced in a less general and slightly different setting in \cite{BaudoinFeng2016}. The integration by part formulas for the intrinsic Malliavin derivative are new. We point out a significant difference of our techniques from what have been used in  \cite{Baudoin2017b, BaudoinEMS2014, BaudoinFeng2016}.  Namely, we shall mostly make use of the adjoint connection $\widehat{\nabla}^\varepsilon$ instead of the Bott connection. Below we summarize important properties of the connection $\widehat{\nabla}^\varepsilon$ which will be used extensively in the sequel.

\begin{rem}[Properties of the adjoint connection] Let $\widehat{\nabla}^\varepsilon$ be the adjoint connection defined by Equation \ref{e.adjointconn}. Then it satisfies the following properties.

\begin{itemize}

\item The adjoint connection is \emph{metric}, that is, $\widehat{\nabla}^\varepsilon g_\varepsilon =0$;

\item The adjoint connection is \emph{horizontal}, that is, if $X \in \Gamma^\infty(\mathcal{H})$ and $Y \in \Gamma^\infty(\M)$ then $\widehat{\nabla}_Y^\varepsilon X \in \Gamma^\infty(\mathcal{H})$;

 \item The torsion tensor $\widehat{T}^\varepsilon$ of $\widehat{\nabla}^\varepsilon$ is \emph{skew-symmetric}, that is, it satisfies  B.~Driver's  skew-symmetry condition (see \cite[Definition 8.1]{Driver1992b}) as follows.  For $X,Y,Z \in \Gamma^\infty(\M)$
\[
\langle \widehat{T}^\varepsilon (X, Y), Z \rangle_\varepsilon =-\langle \widehat{T}^\varepsilon (X, Z), Y \rangle_\varepsilon.
\]
\end{itemize}
The latter can be seen from Equation \eqref{e.4.3}
\[
\widehat{T}^\varepsilon (X,Y)=T(X,Y)-\frac{1}{\varepsilon} J_Y X+\frac{1}{\varepsilon} J_X Y
\]
and the definition of $J$. In addition to \cite{Driver1992b} we refer to \cite{ElworthyLeJanLiBook1999} for more details on adjoint connections  and connections with skew-symmetric torsion, including examples.
\end{rem}

Next recall that a stochastic parallel transport on forms can be defined following  \cite[p. 50]{HsuEltonBook}.

\begin{notation}
Let $\widetilde{\nabla}$ be a general connection on $\M$, and $\left\{ M_s \right\}_{0 \leqslant s \leqslant 1}$ be a semimartingale on $\M$. We denote by
\[
\widetilde{\para}_{0,s} : T_{M_0} \M \to T_{M_s} \M
\]
the \emph{stochastic parallel transport} of vector fields along the paths of $\left\{ M_s \right\}_{0 \leqslant s \leqslant 1}$. Then by duality we can define the stochastic parallel transport on one-forms as follows. We have
\[
\widetilde{\para}_{0,s}: T_{M_s}^\ast \M \to T_{M_0}^\ast \M
\]
such that for $\alpha \in T_{M_s}^\ast \M$
\begin{equation}\label{e.SPTforms}
\langle \widetilde{\para}_{0,s} \alpha, v \rangle =\langle \alpha, \widetilde{\para}_{0, s} v\rangle, \quad v \in T_{M_0} \M.
\end{equation}
\end{notation}
In particular, the stochastic parallel transport for the adjoint connection  $\widehat{\nabla}^\varepsilon = \nabla+\frac{1}{\varepsilon}J$ along the paths of the horizontal Brownian motion $\left\{ W_s \right\}_{0 \leqslant s \leqslant 1}$ will be denoted by  $\widehat{\Theta}^\varepsilon_{s}$. Since  the adjoint connection $\widehat{\nabla}^\varepsilon$ is horizontal,  the map $\widehat{\Theta}^\varepsilon_{s}: T_{x} \M \to T_{W_s} \M$ is an isometry that preserves the horizontal bundle, that is, if $u \in \mathcal{H}_{x}$, then $\widehat{\Theta}^\varepsilon_{s} u \in \mathcal{H}_{W_t}$.  We see then that the anti-development of $\left\{ W_s \right\}_{0 \leqslant s \leqslant 1}$ defined as
 \[
 B_s:=\int_0^s (\widehat{\Theta}^\varepsilon_{r})^{-1} \circ dW_r,
 \]
is a Brownian motion in the horizontal space $\mathcal{H}_{x}$.

\begin{rem}\label{r.EpsilonInd}
Observe that on one-forms the process $\widehat{\Theta}^\varepsilon_{s}:T_{W_s}^*\mathbb{M}\rightarrow T^*_{x}\mathbb{M}$ is a solution to the following covariant Stratonovich stochastic differential equation
\begin{align*}
 d[\widehat{\Theta}^\varepsilon_{s} \alpha(W_s)] & =\widehat{\Theta}^\varepsilon_{s} \widehat{\nabla}^\varepsilon_{\circ dW_s} \alpha (W_s),
 \end{align*}
where $\alpha$ is any smooth one-form. Since $\widehat{\nabla}^\varepsilon_{\circ dW_s} =\nabla_{\circ dW_s} +\frac{1}{\varepsilon} J_{\circ dW_s} =\nabla_{\circ dW_s}$, we deduce that $\widehat{\Theta}^\varepsilon$ is actually independent of $\varepsilon$ and is therefore also the stochastic  parallel transport for the Bott connection. As a consequence, the Brownian motion $\left\{ B_s \right\}_{0 \leqslant s \leqslant 1}$ and its filtration  are also independent of the particular choice of $\varepsilon$.
 \end{rem}

We define a \emph{damped parallel transport} $\tau^\varepsilon_s:T_{W_s}^*\mathbb{M}\rightarrow T^*_{x}\mathbb{M}$  by the formula
\begin{equation}\label{tau=M Theta}
\tau^{\varepsilon}_s=\mathcal{M}_{s}^{\varepsilon}\Theta_{s}^{\varepsilon},
\end{equation}
where the process $\Theta_s^{\varepsilon}: T_{W_s}^{*}\mathbb{M}\rightarrow T_{x}^{*}\mathbb{M}$ is the  stochastic parallel transport of one-forms with respect to the connection $\nabla^\varepsilon=\nabla -\mathfrak{T}^{\varepsilon}$  along the paths of $\left\{ W_s \right\}_{0\leqslant s \leqslant 1}$.
The multiplicative functional $ \mathcal{M}_s^{\varepsilon}: T_{x}^{*}\mathbb{M} \to T_{x}^{*}\mathbb{M}$, $s \geqslant 0$,  is defined as the solution to the following ordinary differential equation
\begin{align}\label{multiplicative function M_t}
& \frac{d\mathcal{M}_s^{\varepsilon}}{ds}=-\frac{1}{2}\mathcal{M}_s^{\varepsilon}\Theta_s^{\varepsilon} \left(\frac{1}{\varepsilon}\mathbf{J}^2-\frac{1}{\varepsilon} \delta_\mathcal{H} T+\mathfrak{Ric}_{\mathcal{H}} \right)(\Theta_s^{\varepsilon})^{-1},
\\
& \mathcal{M}_0^{\varepsilon}=\mathbf{Id}. \notag
\end{align}

Observe that the process $\tau^\varepsilon_s:T_{W_s}^*\mathbb{M}\rightarrow T^*_{x}\mathbb{M}$ is a solution of the following covariant Stratonovich stochastic differential equation
\begin{align}\label{tau_t}
& d[\tau^\varepsilon_s \alpha(W_s)]
\\
&=\tau^\varepsilon_s\left( \nabla_{\circ dW_s}-\mathfrak{T}_{\circ dW_t}^{\varepsilon}-\frac{1}{2} \left( \frac{1}{ \varepsilon}\mathbf{J}^2-\frac{1}{\varepsilon} \delta_\mathcal{H} T+
\mathfrak{Ric}_{\mathcal{H}}\right)ds\right) \alpha(W_s), \notag
\\
& \tau_0=\mathbf{Id}, \notag
\end{align}
where $\alpha$ is any smooth one-form.

Also observe that $\mathcal{M}_s^{\varepsilon}$ is invertible and that its inverse is  the solution of the following ordinary differential equation
\begin{align}
 \frac{d(\mathcal{M}_s^{\varepsilon})^{-1}}{ds}=\frac{1}{2}\Theta_s^{\varepsilon} \left(\frac{1}{\varepsilon}\mathbf{J}^2-\frac{1}{\varepsilon} \delta_\mathcal{H} T+\mathfrak{Ric}_{\mathcal{H}} \right)(\Theta_s^{\varepsilon})^{-1} (\mathcal{M}_s^{\varepsilon})^{-1}.
\end{align}
In particular, it implies that $\tau^\varepsilon_s$ is invertible.

\subsubsection{ Malliavin and directional derivatives}

We recall that the horizontal Wiener measure on $W_{x_{0}}\left( \M \right)$ is defined as the distribution of the horizontal Brownian motion. The coordinate process on $W_{x_{0}}\left( \M \right)$ as before is denoted by $\left\{ w_s \right\}_{ 0 \leqslant s \leqslant 1}$.

\begin{definition}\label{cylinder} A function $F: W_{x_{0}}\left( \M \right) \rightarrow \mathbb{R}$ is called a \emph{$C^{k}$-cylinder function} if there exists a partition

\[
\pi:=\{0=s_0<s_1<s_2< \cdots <s_n\leqslant  1\}
\]
of the interval $[0,1]$ and $f\in C^{k}(\mathbb{M}^{n})$ such that
\begin{equation}\label{cylinder function}
F\left( w \right) = f \left( w_{s_{1}}, . . ., w_{s_{n}} \right) \text{ for all } w \in W_{x_{0}}\left( \M \right).
 \end{equation}
The function $F$ is called a \emph{smooth cylinder function} on $W_{x_{0}}\left( \M \right)$, if there exists a partition $\pi$ and $f\in C^{\infty}(\mathbb{M}^{n})$ such that \eqref{cylinder function} holds.

We denote by $\mathcal{F}C^{k} \left( W_{x_{0}}\left( \M \right)\right)$ the space of $C^{k}$-cylinder functions,  and by $\mathcal{F}C^{\infty} \left(W_{x_{0}}\left( \M \right)\right)$
the space of $C^{\infty}$-cylinder functions.
\end{definition}

\begin{rem}
Note that the representation \eqref{cylinder function} of a cylinder function is not unique. However, let $F \in \mathcal{F}C^{\infty} \left( W_{x_{0}}\left( \M \right)\right)$ and $n \geqslant 0$ be the minimal $n$ such that there exists a partition
\[
\pi:=\{0=s_0<s_1<s_2< \cdots <s_n\leqslant  1\}
\]
of the interval $[0,1]$ and $f\in C^{k}(\mathbb{M}^{n})$ such that
\begin{equation}
F\left( w \right) = f \left( w_{s_{1}}, . . ., w_{s_{n}} \right) \text{ for all } w \in W_{x_{0}}\left( \M \right).
 \end{equation}\label{minimal rep}
In that case, if
 \[
\tilde{\pi}=\{0=\tilde{s}_0<\tilde{s}_1<\tilde{s}_2< \cdots <\tilde{s}_n\leqslant  1\}
\]
is another partition of the interval $[0,1]$ and $\tilde{f}\in C^{k}(\mathbb{M}^{n})$ is such that
\begin{equation*}
F\left( w \right) = \tilde{f} \left( w_{\tilde{s}_{1}}, . . ., w_{\tilde{s}_{n}} \right) \text{ for all } w \in W_x (\M),
 \end{equation*}
then $\pi=\tilde{\pi}$ and $f=\tilde{f}$. Indeed, since
 \[
  f \left( w_{s_{1}}, . . ., w_{s_{n}} \right) =\tilde{f} \left( w_{\tilde{s}_{1}}, . . ., w_{\tilde{s}_{n}} \right)
 \]
we first deduce that $s_1=\tilde{s}_1$. Otherwise $d_1 f=0$ or $d_1 \tilde{f}=0$, where $d_1$ denotes the differential with respect to the first component. This contradicts the fact that $n$ is minimal. Similarly, $s_2=\tilde{s}_2$ and more generally $s_k=\tilde{s}_k$. The representation \eqref{minimal rep} will be referred to as the \emph{minimal representation} of $F$.
\end{rem}
We now turn to the definition of directional derivative on the horizontal path space.

\begin{definition}
Let $F=f \left( w_{s_{1}}, . . ., w_{s_{n}} \right)  \in \mathcal{F}C^{\infty} \left( W_x (\M)\right)$. For an $\mathcal{F}$-adapted and $T_x\M$-valued  semimartingale $(v(s))_{0 \leqslant s \leqslant 1}$ such that $v(0)=0$, we define the directional derivative
\[
\mathbf{D}_v F=\sum_{i=1}^n \left\langle d_if(W_{s_1},\cdots, W_{s_n}),  \widehat{\Theta}_{s_i}^{\varepsilon} v(s_i)  \right\rangle
\]
\end{definition}

\begin{definition}\label{gradient definition}
 For  $F=f \left( w_{s_{1}}, . . ., w_{s_{n}} \right)  \in \mathcal{F}C^{\infty} \left( W_x (\M))\right)$ we define the \emph{damped Malliavin derivative} by
\[
\widetilde{D}_s^{\varepsilon}F:=\sum_{i=1}^n\mathbf{1}_{[0, s_i]}(s)(\tau_s^{\varepsilon})^{-1}\tau_{s_i}^{\varepsilon}d_if(W_{s_1}, \cdots, W_{s_n}), \quad  0 \leqslant s \leqslant 1.
\]
\end{definition}
Observe that from this definition $\widetilde{D}_s^{\varepsilon}F \in T_{W_s}^\ast \M$.

\begin{rem}\label{jkmn}
Note that the directional derivative $\mathbf{D}$ is independent of $\varepsilon$, but the damped Malliavin derivative depends on $\varepsilon$. In addition, both the directional derivatives and damped Malliavin derivatives are independent of the representation of  $F$. Indeed, let $F=f \left( w_{s_{1}}, . . ., w_{s_{n}} \right)$ be the minimal representation of $F$. If $\tilde{f} \left( w_{\tilde{s}_{1}}, . . ., w_{\tilde{s}_{N}} \right)$ is another representation of $F$, then for every $1 \leqslant j \leqslant N$, we have either that there exists $i$ such that $s_i=\tilde{s}_j$ in which case $d_i f=d_j \tilde{f}$, or  for all $i$, $s_i \not= \tilde{s}_j$ in which case $d_j \tilde{f}=0$.

\end{rem}

Before we can formulate the main result, we need to define an analog of the Cameron-Martin subspace.


\begin{definition}\label{d.CameronMartinProcess}
An $\mathcal{F}_{s}$-adapted absolutely continuous  $\mathcal{H}_{x}$-valued process $\left\{ \gamma(s)\right\}_{0\leqslant s \leqslant 1}$ such that $\gamma(0)=0$ and $\mathbb{E}_x \left( \int_0^1 \| \gamma^{\prime}(s) \|_{\mathcal{H}}^2 ds \right)<\infty$ will be called a \emph{horizontal Cameron-Martin process}. 
\end{definition}

\begin{definition}\label{admissible}
Suppose $\left\{ v(s) \right\}_{0\leqslant s \leqslant 1}$ is an $\mathcal{F}_{s}$-adapted  $T_x\M$-valued continuous semimartingale such that $v(0)=0$ and $\mathbb{E}_x \left( \int_0^1 \| v(s) \|^2 ds \right)<\infty$. We call $\left\{ v(s) \right\}_{0\leqslant s \leqslant 1}$ a \emph{tangent process} if the process
 \[
 v(s)-\int_0^s (\widehat{\Theta}_r^{\varepsilon})^{-1} T ( \widehat{\Theta}_sr^{\varepsilon} \circ dB_r, \widehat{\Theta}_r^{\varepsilon} v(r))
 \]
is a horizontal Cameron-Martin process. 
\end{definition}

\begin{remark}
By Remark \ref{r.EpsilonInd} the stochastic parallel transport $\widehat{\Theta}_s^{\varepsilon}$ is independent of $\varepsilon$, therefore the notion of a tangent process is itself independent of $\varepsilon$ as well.

\end{remark}

\begin{remark}
As the torsion $T$ is a vertical tensor, then an $\mathcal{F}_{s}$-adapted $T_x\M$-valued continuous semimartingale $\left\{v(s)\right\}_{0\leqslant s \leqslant 1}$ such that

\[
\mathbb{E}_x \left( \int_0^1 \| v(s) \|^2 ds \right)<\infty, \hskip0.1in v(0)=0
\]
is in $T W_\mathcal{H}(\mathbb{M})$ if and only if
\begin{enumerate}
\item The horizontal part $v_\mathcal{H}$ is a horizontal Cameron-Martin process;
\item The vertical part $v_\mathcal{V}$ is given by
\[
v_\mathcal{V} (s)=\int_0^s (\widehat{\Theta}_r^{\varepsilon})^{-1} T ( \widehat{\Theta}_r^{\varepsilon} \circ dB_r, \widehat{\Theta}_r^{\varepsilon} v_\mathcal{H} (r)).
\]
\end{enumerate}

\end{remark}
The main results of this section are the following two theorems.

\begin{thm}[Integration by parts for the damped Malliavin derivative] \label{IBP} { \ }
Suppose $F \in \mathcal{F}C^{\infty} \left( W_x (\M)\right)$ and $\gamma$ is a tangent process, then

\begin{align}\label{JHKL}
\mathbb{E}_x\left( \int_0^1\langle \widetilde{D}_s^{\varepsilon}F,\widehat{\Theta}^\varepsilon_{s} \gamma^{\prime}(s)\rangle ds \right)=\mathbb{E}_x\left( F \int_0^1\langle \gamma^{\prime}(s), dB_s\rangle_{\mathcal{H}}\right).
\end{align}

\end{thm}

\begin{thm}[Integration by parts for the directional derivatives]\label{IBP2} { \ }
Suppose  $F \in \mathcal{F}C^{\infty} \left( W_x (\M)\right)$ and $v$ is a tangent process, then

\begin{align*}
  \mathbb{E}_x\left( \mathbf{D}_v F \right)
 =\mathbb{E}_x \left( F \int_0^1 \left\langle v_\mathcal{H}'(s)+ \frac{1}{2} (\widehat{\Theta}_s^{\varepsilon})^{-1} \mathfrak{Ric}_{\mathcal{H}}   \widehat{\Theta}_s^{\varepsilon} v_{\mathcal{H}} (s), dB_t\right\rangle_{\mathcal{H}}  \right).
\end{align*}

\end{thm}
Even though these two integration by parts formulas seem similar, they are quite different in nature. The damped derivative is used to derive gradient bounds and functional inequalities on the path space (e.g. \cite{Baudoin2017b, BaudoinFeng2016}). The directional derivative, however, is more related to quasi-invariance properties such as in Section \ref{s.QI}, and the expression

\[
\int_0^1 \left\langle v_\mathcal{H}^{\prime}(s)+ \frac{1}{2} (\widehat{\Theta}_s^{\varepsilon})^{-1} \mathfrak{Ric}_{\mathcal{H}}   \widehat{\Theta}_s^{\varepsilon} v_{\mathcal{H}}(s), dB_s\right\rangle_{\mathcal{H}}
\]
can  be viewed as a horizontal divergence on the path space.

The remainder of the section is devoted to proving Theorem \ref{IBP} and Theorem \ref{IBP2}. We adapt the techniques from the Markovian stochastic calculus developed by Fang-Malliavin \cite{FangMalliavin1993} and E.~Hsu \cite{Hsu1997c} in the Riemannian case to our setting.

\subsubsection{Gradient formula}
In this preliminary section we recall the gradient formula for the semigroup $P_s$. In the case the Yang-Mills condition is satisfied, that is,  the horizontal divergence $\delta_\mathcal{H} T=0$, the operator $\square_\varepsilon$ is essentially self-adjoint on $L^2\left( \M, g_\varepsilon, \mu \right)$ equipped with  inner product on sections, i.e. $\int_\M \langle\cdot,\cdot\rangle_{\varepsilon}d\mu$, and the gradient representation  was first proved in \cite{Baudoin2017b}.

\begin{lem}[Theorem 4.6 and Corollary 4.7  in \cite{Baudoin2017b}, Theorem 2.7 in \cite{Grong}]
\label{l.5.4}
For  $f \in C^\infty(\M)$, the process
\begin{equation}\label{e.5.11}
N_s=\tau^\varepsilon_s(dP_{1-s} f) (W_s), \quad 0 \leqslant s \leqslant 1,
\end{equation}
is a  martingale, where $dP_{1-s} f$ denotes the exterior derivative of the function $P_{1-s} f$. As a consequence, for every $0 \leqslant s \leqslant 1$,
\begin{align}\label{rep}
dP_sf (x)=\mathbb{E}_x ( \tau^\varepsilon_s df (W_s) ).
\end{align}
\end{lem}

\begin{proof}
From It\^o's formula and the definition of $\tau^\varepsilon$, we have
\begin{align*}
& dN_s=
\\
& \tau^\varepsilon_s \left( \nabla_{\circ dW_s}-\mathfrak{T}_{\circ dW_s}^{\varepsilon}-\frac{1}{2} \left( \frac{1}{ \varepsilon}\mathbf{J}^2-\frac{1}{\varepsilon} \delta_\mathcal{H} T + \mathfrak{Ric}_{\mathcal{H}}\right)ds\right) (dP_{1-s} f) (W_s)
\\
& +\tau^\varepsilon_s \frac{d}{ds} (dP_{1-s} f) (W_s) ds.
\end{align*}
We now see that
\begin{align*}
\frac{d}{ds} (dP_{1-s} f) =-\frac{1}{2} dP_{1-s} Lf=-\frac{1}{2}dLP_{1-s} f=-\frac{1}{2} \square_\varepsilon dP_{1-s}f,
\end{align*}
where we used Theorem \ref{Weitzenbock}. Observe that the bounded variation part of \[
\tau^\varepsilon_s \left( \nabla_{\circ dW_s}-\mathfrak{T}_{\circ dW_s}^{\varepsilon}-\frac{1}{2} \left( \frac{1}{ \varepsilon}\mathbf{J}^2-\frac{1}{\varepsilon} \delta_\mathcal{H} T+
\mathfrak{Ric}_{\mathcal{H}}\right)ds\right) (dP_{1-s} f) (W_s)
\]
is given by $\frac{1}{2} \tau^\varepsilon_s \square_\varepsilon dP_{1-s}f (W_s) ds$ which cancels out with the expression
\[
\tau^\varepsilon_s \frac{d}{ds} (dP_{1-s} f) (W_s) ds
\]
in the first equation. The martingale property follows from a bound  similar to  \cite[Lemma 4.3]{Baudoin2017b} or \cite[Theorem 2.7]{GrongThalmaier2016a, GrongThalmaier2016b}.
\end{proof}

\subsubsection{Integration by parts formula for the damped Malliavin derivative}

We prove Theorem \ref{IBP} in this section. Some of the key arguments may be found in \cite{Baudoin2017b, BaudoinFeng2016}, however since our framework is more general here (for example, we do not assume the Yang-Mills condition that the horizontal divergence $\delta_\mathcal{H} T=0$) and we now use the adjoint connection $\widehat{\nabla}^\ve$ instead of the Bott connection, for the sake of self-containment, we give a complete proof.

\begin{lem}\label{IPP}
For $f \in C^\infty(\M)$, and $\gamma$ horizontal Cameron-Martin process
\begin{align*}
& \mathbb{E}_x \left( f(W_1) \int_0^1 \langle \gamma'(s),dB_s\rangle_{\mathcal{H}}  \right)=
\\
& \mathbb{E}_x \left(\left\langle  \tau^\varepsilon_1 df(W_1),\int_0^1  (\tau^{\varepsilon,*}_s)^{-1}  \widehat{\Theta}^\varepsilon_{s} \gamma'(s) ds  \right\rangle \right).
\end{align*}
 \end{lem}

 \begin{proof}
Consider again the martingale process $N_s$ defined by \eqref{e.5.11}. We have then for $f \in C^\infty(\M)$
\begin{align*}
&\mathbb{E}_x \left( f(W_s) \int_0^s \langle \gamma'(r), dB_r\rangle_{\mathcal{H}}  \right)=\mathbb{E}_x \left( f(W_s) \int_0^s \langle \widehat{\Theta}^\varepsilon_{r} \gamma'(r),\widehat{\Theta}^\varepsilon_{r}dB_r\rangle_{\mathcal{H}}  \right) \\
 &=\mathbb{E}_x \left( ( f(W_s) -\mathbb{E}_x \left( f(W_s)\right)) \int_0^s \langle \widehat{\Theta}^\varepsilon_{r} \gamma'(r),\widehat{\Theta}^\varepsilon_{r}dB_r\rangle_{\mathcal{H}}  \right) \\
 &=\mathbb{E}_x \left(\int_0^s \langle dP_{s-r}f (W_r), \widehat{\Theta}^\varepsilon_{r} dB_r \rangle  \int_0^s \langle \widehat{\Theta}^\varepsilon_{r} \gamma'(r),\widehat{\Theta}^\varepsilon_{r}dB_r\rangle_{\mathcal{H}}  \right) \\
 & =\mathbb{E}_x \left(\int_0^s \langle dP_{s-r}f (W_r), \widehat{\Theta}^\varepsilon_{r} \gamma'(r) \rangle dr  \right) \\
 &=\mathbb{E}_x \left(\int_0^s \langle \tau_r^\varepsilon dP_{s-r}f (W_r), (\tau^{\varepsilon,*}_r)^{-1}  \widehat{\Theta}^\varepsilon_{r} \gamma'(r) \rangle dr  \right)  \\
 &=\mathbb{E}_x \left(\int_0^s \langle N_r, (\tau^{\varepsilon,*}_r)^{-1}  \widehat{\Theta}^\varepsilon_{r} \gamma'(r) \rangle dr  \right) \\
 &=\mathbb{E}_x \left(\left\langle  N_s,\int_0^s  (\tau^{\varepsilon,*}_r)^{-1} \widehat{\Theta}^\varepsilon_{r} \gamma'(r) dr  \right\rangle \right),
\end{align*}
where we integrated by parts in the last equality.
 \end{proof}

\begin{rem}\label{conditional H}
A similar proof as above actually yields that for $f \in C^\infty(\M)$, $\gamma$ horizontal Cameron-Martin process and $0 \leqslant s \leqslant 1$,
\begin{align*}
& \mathbb{E}_x \left( f(W_1) \int_s^1 \langle \gamma^{\prime}(r), dB_r\rangle_{\mathcal{H}} \mid \mathcal{F}_s \right)=
\\
& \mathbb{E}_x \left(\left\langle  \tau^\varepsilon_1 df(W_1), \int_s^1  (\tau^{\varepsilon,*}_r)^{-1}  \widehat{\Theta}^\varepsilon_{r} \gamma^{\prime}(r) dr  \right\rangle  \mid \mathcal{F}_s \right).
\end{align*}
\end{rem}
Lemma \ref{IPP} shows that integration by parts formula \eqref{JHKL} holds for cylinder functions of the type $F=f(W_s)$. We now turn to the proof of Theorem \ref{IBP} by using induction on $n$ in a representation of a cylinder function $F$. To run the induction argument we need the following fact.

\begin{prop}\label{gradient of expectation formula}
Let  $F=f(W_{s_1}, \cdots, W_{s_n}) \in  \mathcal{F}C^{\infty} \left( W_x (\M)\right)$. We have
\[
d \mathbb{E}_x(F)=\mathbb{E}_x\left(\sum_{i=1}^n\tau_{s_i}^{\varepsilon}d_if(W_{s_1}, \cdots, W_{s_n})\right).
\]
\end{prop}
\begin{proof}
We will proceed  by induction on $n$. Consider a cylinder function $F=f(W_{s_1}, \cdots, W_{s_n})$. For $n=1$ the statement follows from  Lemma \ref{l.5.4}, which implies that
\[
d\mathbb{E}_x(f(W_{s_1}))=dP_{s_1}f (x)=\mathbb{E}_x ( \tau^\varepsilon_{s_1} df(W_{s_1}) ).
\]
Now we assume that the claim holds for any cylinder function of the form $F=f(W_{s_1}, \cdots, W_{s_k})$ for any $k \leqslant n-1$. By the Markov property we have
\[
\mathbb{E}_x(F)=\mathbb{E}_x(\mathbb{E}( F \mid \mathcal{F}_{s_1} ))=\mathbb{E}_x(g(W_{s_1})),
\]
where $g(y)=\mathbb{E}_y(f(y, W_{s_2-s_1}, \cdots, W_{s_n-s_1}))$.
Therefore
\[
d\mathbb{E}_x(F)=\mathbb{E}(\tau_{s_1}^{\varepsilon}dg(W_{s_1})).
\]
By using the induction hypothesis, we obtain
\begin{align*}
dg(y)&=\mathbb{E}_y(d_1f(y, W_{s_2-s_1}, \cdots, W_{s_n-s_1}))
 +
 \\
 & \mathbb{E}_y\left(\sum_{i=2}^n\tau^{\varepsilon}_{s_i-s_1}d_if(y, W_{s_2-s_1}, \cdots, W_{t_n-t_1})\right) \\
 &=\mathbb{E}_y\left(\sum_{i=1}^n\tau^{\varepsilon}_{s_i-s_1}d_if(y,W_{s_2-s_1}, \cdots, W_{s_n-s_1})\right).
\end{align*}
By the multiplicative property of $\tau^{\varepsilon}$ and the Markov property of $W$ we have
\begin{align*}
&\mathbb{E}_{W_{s_1}}\left(\tau^{\varepsilon}_{s_i-s_1}d_if(y,W_{s_2-s_1}, \cdots, W_{s_n-s_1})\right) =
\\
& (\tau_{s_1}^{\varepsilon})^{-1} \mathbb{E}\left( \tau_{s_i}^{\varepsilon}d_if(W_{s_1}, \cdots, W_{s_n}) \mid \mathcal{F}_{s_1} \right).
\end{align*}
Therefore we conclude
\begin{align*}
& d\mathbb{E}_x(F)=\mathbb{E}_x\left(\sum_{i=1}^n\tau_{s_i}^{\varepsilon}d_if(W_{s_1}, \cdots, W_{s_n})\right).
\end{align*}
\end{proof}

\begin{rem} As expected, the expression
\[
\mathbb{E}_x\left(\sum_{i=1}^n\tau_{s_i}^{\varepsilon}d_if(W_{s_1}, \cdots, W_{s_n})\right)
\]
is independent of the choice of the representation of the cylinder function $F$, as follows from  Remark \ref{jkmn}.
\end{rem}

\begin{proof}[Proof of Theorem \ref{IBP}] We use induction on $n$ in a representation of  the cylinder function $F$ . More precisely,
we would like to show that for any $F=f(W_{s_1}, \cdots, W_{s_n}) \in \mathcal{F}C^{\infty} \left( W_x (\M)\right)$ and $s \leqslant s_{1}$ we have
\begin{align}
&  \mathbb{E}_x\left( \left.F \int_{s}^{s_n}\langle \gamma^{\prime}(r), dB_r \rangle_{\mathcal{H}} \right| \mathcal{F}_{s}\right)=
\label{e.IbPinduction}
 \\
& \mathbb{E}_{x} \left(\sum_{i=1}^n \langle d_if(W_{s_1}, \cdots, W_{s_n}),\tau_{s_i}^{\varepsilon,*} \int_{s}^{s_i} (\tau_r^{\varepsilon,*})^{-1} \gamma'(r) dr \rangle \mid \mathcal{F}_{s}\right). \notag
\end{align}

The case $n=1$ is  Lemma \ref{IPP} and Remark \ref{conditional H}. Assume that \eqref{e.IbPinduction} holds for any cylinder function $F$ represented by a partition of size $n-1$ for $n \geqslant 2$. Let  $F=f(W_{s_1}, \cdots, W_{s_n}) \in \mathcal{F}C^{\infty} \left( W_x (\M)\right)$. We have for $s  \leqslant s_1$,
\begin{align*}
 &\mathbb{E}_x\left( F \int_s^1\langle \gamma^{\prime}(r), dB_r\rangle_{\mathcal{H}} \mid \mathcal{F}_s\right)
 \\
=&\mathbb{E}_x\left( F \int_s^{s_n}\langle \gamma^{\prime}(r), dB_r\rangle_{\mathcal{H}} \mid \mathcal{F}_s \right)
\\
 =&\mathbb{E}_x\left( F \int_s^{s_1}\langle \gamma^{\prime}(r), dB_r\rangle_{\mathcal{H}}  \mid \mathcal{F}_s  \right)+\mathbb{E}_x\left( F \int_{s_1}^{s_n}\langle \gamma^{\prime}(r), dB_r\rangle_{\mathcal{H}}  \mid \mathcal{F}_s  \right)
 \\
 =&\mathbb{E}_x\left( \mathbb{E}_x (F \mid \mathcal{F}_{s_1}) \int_s^{s_1}\langle \gamma^{\prime}(r), dB_r\rangle_{\mathcal{H}}  \mid \mathcal{F}_s  \right)+
 \\
 & \mathbb{E}_x\left( \mathbb{E}_x\left( F \int_{s_1}^{s_n}\langle \gamma^{\prime}(r), dB_r\rangle_{\mathcal{H}} | \mathcal{F}_{s_1}\right) \mid \mathcal{F}_s  \right).
\end{align*}
 By the Markov property we have
\[
 \mathbb{E}_x (F \mid \mathcal{F}_{s_1})=g(W_{s_1}),
\]
where $g(y)=\mathbb{E}_y(f(y, W_{s_2-s_1}, \cdots, W_{s_n-s_1}))$. Thus by Lemma \ref{IPP} and Remark \ref{conditional H}
\begin{align*}
& \mathbb{E}_x\left( \ \mathbb{E}_x (F \mid \mathcal{F}_{s_1}) \int_s^{s_1}\langle \gamma^{\prime}(r), dB_r\rangle_{\mathcal{H}} \mid \mathcal{F}_s \right)
=
\\
& \mathbb{E}_x\left( g(W_{s_1}) \int_s^{s_1}\langle \gamma^{\prime}(r), dB_r\rangle_{\mathcal{H}} \mid \mathcal{F}_s \right)=
\\
&\mathbb{E}_x \left(\left\langle   dg (W_{s_1}), (\tau^\varepsilon_{s_1})^*\int_s^{s_1}  (\tau^{\varepsilon,*}_r)^{-1}  \widehat{\Theta}^\varepsilon_{r}\gamma^{\prime}(r) dr  \right\rangle \mid \mathcal{F}_s \right)
\end{align*}
Now according to Proposition \ref{gradient of expectation formula}
\begin{align*}
dg(y)=\mathbb{E}_y\left(\sum_{i=1}^n\tau^{\varepsilon}_{s_i-s_1}d_if(y, W_{s_2-s_1}, \cdots, W_{s_n-s_1})\right).
\end{align*}
Using the fact that
\begin{align*}
& \mathbb{E}_{W_{s_1}}\left(\tau^{\varepsilon}_{s_i-s_1}d_if(y, W_{s_2-s_1}, \cdots, W_{s_n-s_1}) \right)
 =
 \\
& (\tau_{s_1}^{\varepsilon})^{-1} \mathbb{E}_x\left( \tau_{s_i}^{\varepsilon}d_if(W_{s_1}, \cdots, W_{s_n}) \mid \mathcal{F}_{s_1} \right),
\end{align*}
we conclude
\begin{align*}
  &\mathbb{E}_x\left(  \mathbb{E}_x (F \mid \mathcal{F}_{s_1}) \int_s^{s_1}\langle \gamma^{\prime}(r), dB_r\rangle_{\mathcal{H}} \mid \mathcal{F}_s \right) \\
  = & \mathbb{E}_x\left(\sum_{i=1}^n \langle d_if(W_{s_1}, \cdots, W_{s_n}),\tau_{s_i}^{\varepsilon,*} \int_s^{s_1} (\tau_r^{\varepsilon,*})^{-1}\widehat{\Theta}^\varepsilon_{r} \gamma^{\prime}(r) dr \rangle \mid \mathcal{F}_s \right).
\end{align*}
Using the induction hypothesis that \eqref{e.IbPinduction} holds for $n-1$ we see that
\begin{align*}
& \mathbb{E}_x\left( \left.F \int_{s_1}^{s_n}\langle \gamma^{\prime}(r), dB_r\rangle_{\mathcal{H}} \right| \mathcal{F}_{s_1}\right)=
 \\
& \mathbb{E}_{x} \left(\sum_{i=1}^n \langle d_if(W_{s_1},\cdots, W_{s_n}),\tau_{s_i}^{\varepsilon,*} \int_{s_1}^{s_i} (\tau_r^{\varepsilon,*})^{-1} \gamma'(r) dr \rangle \mid \mathcal{F}_{s_1}\right).
\end{align*}

\end{proof}

\subsubsection{Integration by parts formula for the directional derivatives}

In this section we prove Theorem \ref{IBP2}. One of the main ingredients  B.~Driver used in \cite{Driver1992b} in the Riemannian case was the orthogonal invariance of the Brownian motion to filter out redundant noise. As a complement to Lemma \ref{IPP}, we first prove the following result.

\begin{lem}\label{IPP2}
Let $\left\{ \mathcal{O}_s \right\}_{0 \leqslant s \leqslant 1}$ be a continuous $\mathcal{F}$-adapted process taking values in the space of skew-symmetric endomorphisms of $\mathcal{H}_x$  such that $\mathbb{E} \left( \int_0^1 \| \mathcal{O}_s \|^2 ds\right) < \infty$, where $ \| \mathcal{O}_s \|^2 =\mathbf{Tr} (\mathcal{O}_s^* \mathcal{O}_s)$  . For $f \in C^\infty(\M)$, we have
\[
\mathbb{E}_x \left(\left\langle  \tau^\varepsilon_1 df(W_1),\int_0^1  (\tau^{\varepsilon,*}_s)^{-1}  \widehat{\Theta}^\varepsilon_{s} \left(\mathcal{O}_s dB_s-\frac{1}{2} T^\varepsilon_{\mathcal{O}_s} ds \right) \right\rangle \right)=0,
\]
where $T^\varepsilon_{\mathcal{O}_s}$ is the tensor given in a horizontal frame $e_1,\cdots,e_n$ by
\[
T^\varepsilon_{\mathcal{O}_s}=\sum_{i=1}^n  (\widehat{\Theta}^\varepsilon_{s})^{-1}T^\varepsilon (e_i, \widehat{\Theta}^\varepsilon_{s} \mathcal{O}_s  (\widehat{\Theta}^\varepsilon_{s})^{-1} e_i).
\]
 \end{lem}

\begin{proof}
Recall that we considered the following  martingale in \eqref{e.5.11}
\begin{equation*}
N_s=\tau^\varepsilon_s(dP_{1-s} f) (W_s), \quad 0 \leqslant s \leqslant 1.
\end{equation*}
We have then
\begin{align*}
& \mathbb{E}_x \left(\left\langle  \tau^\varepsilon_1 df(W_1),\int_0^1  (\tau^{\varepsilon,*}_s)^{-1}  \widehat{\Theta}^\varepsilon_{s} \mathcal{O}_s dB_s \right\rangle \right)=
\\
& \mathbb{E}_x \left(\left\langle  N_1,\int_0^1  (\tau^{\varepsilon,*}_s)^{-1}  \widehat{\Theta}^\varepsilon_{s} \mathcal{O}_s dB_s \right\rangle \right).
\end{align*}
From the proof of Lemma \ref{l.5.4}, we have
\begin{align*}
& dN_s
\\
& =\tau^\varepsilon_s \left( \nabla_{\circ dW_s}-\mathfrak{T}_{\circ dW_s}^{\varepsilon}-\frac{1}{2} \left( \frac{1}{ \varepsilon}\mathbf{J}^2-\frac{1}{\varepsilon} \delta_\mathcal{H} T + \mathfrak{Ric}_{\mathcal{H}}\right)ds\right) (dP_{1-s} f) (W_s)
\\
& +\tau^\varepsilon_s \frac{d}{ds} (dP_{1-s} f) (W_s) ds
\\
&=\tau^\varepsilon_s \left( \nabla_{\widehat{\Theta}^\varepsilon_{s} dB_s }-\mathfrak{T}_{\widehat{\Theta}^\varepsilon_{s}dB_s}^{\varepsilon} \right) (dP_{1-s} f) (W_s)=\tau^\varepsilon_s\nabla^\varepsilon_{\widehat{\Theta}^\varepsilon_{s} dB_s }dP_{1-s} f (W_s).
\end{align*}
where, as before, $\nabla^\varepsilon$ denotes the connection $\nabla-\mathfrak{T}^\varepsilon$. Let us denote by $\mathbf{Hess}^\varepsilon$ the Hessian for the connection $\nabla^\varepsilon$. One has therefore

\begin{align*}
& \mathbb{E}_x \left(\left\langle  N_1,\int_0^1  (\tau^{\varepsilon,*}_s)^{-1}  \widehat{\Theta}^\varepsilon_{s} \mathcal{O}_s dB_s \right\rangle \right) =
\\
& \mathbb{E}_x \left(\int_0^1 \mathbf{Hess}^\varepsilon P_{1-s} f ( \widehat{\Theta}^\varepsilon_{s} dB_s,  \widehat{\Theta}^\varepsilon_{s} \mathcal{O}_s dB_s) (W_s)    \right)
\end{align*}

Due to the skew symmetry of $\mathcal{O}$ and the fact that for $h \in C^\infty(\M)$, $X,Y \in \Gamma^\infty(\M)$,
\[
\mathbf{Hess}^\varepsilon h (X,Y)  -\mathbf{Hess}^\varepsilon h (Y,X)=T^\varepsilon (X,Y) h,
\]
we deduce

\begin{align*}
& \mathbb{E}_x \left(\left\langle  N_1,\int_0^1  (\tau^{\varepsilon,*}_s)^{-1}  \widehat{\Theta}^\varepsilon_{s} \mathcal{O}_s dB_s \right\rangle \right) =
\\
&
\frac{1}{2} \mathbb{E}_x \left(\int_0^1 \left\langle dP_{1-s} f,  \widehat{\Theta}^\varepsilon_{s} T^\varepsilon_{\mathcal{O}_s}  \right\rangle ds \right)=
\\
&
\frac{1}{2} \mathbb{E}_x \left(\int_0^1 \left\langle N_s, (\tau^{\varepsilon,*}_s)^{-1}  \widehat{\Theta}^\varepsilon_{s}T^\varepsilon_{\mathcal{O}_s}  \right\rangle ds \right).
\end{align*}
Integrating by parts the right hand side yields the conclusion.
\end{proof}

We are now in position to prove the integration by parts formula for cylinder functions of the type $F=f(W_s)$.

\begin{lem}\label{IPP4}
Let $v $ be a tangent process. For $f \in C^\infty (\M)$,
\begin{align*}
& \mathbb{E}_x \left( \left\langle df(W_1), \widehat{\Theta}_1^\varepsilon v (1) \right\rangle \right)=
\\
& \mathbb{E}_x \left( f(W_1) \int_0^1 \left\langle v_\mathcal{H}'(s)+ \frac{1}{2} (\widehat{\Theta}_s^{\varepsilon})^{-1} \mathfrak{Ric}_{\mathcal{H}}   \widehat{\Theta}_s^{\varepsilon} v_{\mathcal{H}} (s), dB_s\right\rangle_{\mathcal{H}} \right).
\end{align*}
\end{lem}

 \begin{proof}
 Let $v $ be a tangent process.  We define
 \[
h(s)= v(s)-\int_0^s (\widehat{\Theta}_r^{\varepsilon})^{-1} T ( \widehat{\Theta}_r^{\varepsilon} \circ dB_r, \widehat{\Theta}_r^{\varepsilon} v(r)).
 \]
By definition of tangent processes, we have that $h=v_\mathcal{H}$ is a  horizontal Cameron-Martin process. By Equation \eqref{e.4.3} we have
\[
\widehat{T}^\varepsilon (\circ dW_s,  \widehat{\Theta}_s^{\varepsilon} v(s) )=T( \widehat{\Theta}_s^{\varepsilon} \circ dB_s,\widehat{\Theta}_s^{\varepsilon} v(s) )-\frac{1}{\varepsilon} J_{\widehat{\Theta}_s^{\varepsilon} v(s) } (\widehat{\Theta}_s^{\varepsilon} \circ dB_s).
\]
Therefore we get
\begin{align*}
& dv(s)+(\widehat{\Theta}_s^{\varepsilon})^{-1}  \left(-\widehat{T}^\varepsilon (\circ dW_s,\cdot) + \frac{1}{2} \left( \frac{1}{ \varepsilon}\mathbf{J}^2-\frac{1}{\varepsilon} \delta^*_\mathcal{H} T +\mathfrak{Ric}_{\mathcal{H}} \right) ds  \right) \widehat{\Theta}_s^{\varepsilon} v(s)
\\
= &dh(s)
+\frac{1}{\varepsilon} (\widehat{\Theta}_s^{\varepsilon})^{-1}J_{\widehat{\Theta}_s^{\varepsilon} v(s) } (\widehat{\Theta}_s^{\varepsilon} \circ dB_s)
\\
+& \frac{1}{2} (\widehat{\Theta}_s^{\varepsilon})^{-1}\left( \frac{1}{ \varepsilon}\mathbf{J}^2-\frac{1}{\varepsilon} \delta^*_\mathcal{H} T +\mathfrak{Ric}_{\mathcal{H}} \right) \widehat{\Theta}_s^{\varepsilon} h(s) ds
\\
 = & dh(s) +\frac{1}{\varepsilon} (\widehat{\Theta}_s^{\varepsilon})^{-1}J_{\widehat{\Theta}_s^{\varepsilon} v(s) }(\widehat{\Theta}_s^{\varepsilon}  dB_s)
 + \frac{1}{2} (\widehat{\Theta}_s^{\varepsilon})^{-1}\left( \mathfrak{Ric}_{\mathcal{H}} \right) \widehat{\Theta}_t^{\varepsilon} h(s) ds.
\end{align*}
In the last computation, the transformation of the Stratonovitch differential
\[
(\widehat{\Theta}_s^{\varepsilon})^{-1}J_{\widehat{\Theta}_s^{\varepsilon} v(s) }(\widehat{\Theta}_s^{\varepsilon} \circ dB_s)
\]
into  It\^o's differential
\[
(\widehat{\Theta}_s^{\varepsilon})^{-1}J_{\widehat{\Theta}_s^{\varepsilon} v(s) }(\widehat{\Theta}_s^{\varepsilon} \ dB_s)
\]
is similar to \eqref{trace J}. It is then a consequence of It\^o's formula that
 \[
 v(s)=(\widehat{\Theta}_s^{\varepsilon})^{-1} \tau_s^{\varepsilon,*} \int_0^s  (\tau_r^{\varepsilon,*} )^{-1} \widehat{\Theta}_r^{\varepsilon} \circ dM_r,
 \]
 where
 \[
 dM_s=dh(s) +\frac{1}{\varepsilon} (\widehat{\Theta}_s^{\varepsilon})^{-1}J_{\widehat{\Theta}_s^{\varepsilon} v(s) }\widehat{\Theta}_s^{\varepsilon}  dB_s+ \frac{1}{2} (\widehat{\Theta}_s^{\varepsilon})^{-1}\left( \mathfrak{Ric}_{\mathcal{H}} \right) \widehat{\Theta}_s^{\varepsilon} h(s) ds.
 \]
Converting the Stratonovich integral into It\^o's integral finally yields
\begin{align*}
& v(s)=(\widehat{\Theta}_s^{\varepsilon})^{-1} \tau_s^{\varepsilon,*} \int_0^s  (\tau_r^{\varepsilon,*} )^{-1} \widehat{\Theta}_r^{\varepsilon}  \left( dh(s) + \frac{1}{2} (\widehat{\Theta}_s^{\varepsilon})^{-1}\left( \mathfrak{Ric}_{\mathcal{H}} \right) \widehat{\Theta}_s^{\varepsilon} h(s) ds
 \right.
\\
& \left.+\mathcal{O}_s dB_s-\frac{1}{2} T^\varepsilon_{\mathcal{O}_s} ds\right),
 \end{align*}
 with
 \[
 \mathcal{O}_s=\frac{1}{\varepsilon} (\widehat{\Theta}_s^{\varepsilon})^{-1}J_{\widehat{\Theta}_s^{\varepsilon} v(s) }\widehat{\Theta}_s^{\varepsilon}.
 \]
 Since $ \mathcal{O}_s$ is a skew-symmetric horizontal endomorphism, one can conclude from Lemmas \ref{IPP} and \ref{IPP2} that
\begin{align*}
& \mathbb{E}_x \left( \left\langle df(W_s), \widehat{\Theta}_t^\varepsilon v (s) \right\rangle \right)
\\
& = \mathbb{E}_x \left( f(W_s) \int_0^1 \left\langle v_\mathcal{H}^{\prime}(s)+ \frac{1}{2} (\widehat{\Theta}_s^{\varepsilon})^{-1} \mathfrak{Ric}_{\mathcal{H}}   \widehat{\Theta}_s^{\varepsilon} v_{\mathcal{H}} (s), dB_s\right\rangle_{\mathcal{H}} \right)
\end{align*}
because $h(s)=v_\mathcal{H} (s)$.
\end{proof}

Now Theorem \ref{IBP2} can be proven using  induction on $n$ in the representation of a cylinder function $F$. The case $n=1$ is Lemma \ref{IPP4}, and showing the induction step is similar to how Theorem \ref{IBP} has been proven, so for the sake of conciseness of the paper, we omit the details. As a direct corollary of Theorem \ref{IBP2}, we obtain the following.

\begin{corollary}
Let $F,G \in \mathcal{F}C^{\infty} \left( W_x (\M)\right)$ and  $v $ be a tangent process. We have
\[
\mathbb{E}_x ( F \mathbf{D}_v G) =\mathbb{E}_x ( G \mathbf{D}^{*}_v F),
\]
where
\[
\mathbf{D}^{*}_v =-\mathbf{D}_v +\int_0^1 \left\langle v_\mathcal{H}'(s)+ \frac{1}{2} (\widehat{\Theta}_s^{\varepsilon})^{-1} \mathfrak{Ric}_{\mathcal{H}}   \widehat{\Theta}_s^{\varepsilon} v_{\mathcal{H}} (s), dB_s \right\rangle_{\mathcal{H}}.
\]
\end{corollary}

\begin{proof}
By Theorem \ref{IBP2}, we have
\[
\mathbb{E}_x (  \mathbf{D}_v (FG))=\mathbb{E}_x \left( FG \int_0^1 \left\langle v_\mathcal{H}'(s)+ \frac{1}{2} (\widehat{\Theta}_s^{\varepsilon})^{-1} \mathfrak{Ric}_{\mathcal{H}}   \widehat{\Theta}_s^{\varepsilon} v_{\mathcal{H}} (s), dB_s \right\rangle_{\mathcal{H}} \right).
\]
Since $\mathbf{D}_v (FG)=F \mathbf{D}_v (G) +G \mathbf{D}_v (F)$, the conclusion follows immediately.
\end{proof}

\subsection{Examples}

\subsubsection{Riemannian submersions}\label{ss.RiemSubmersions}

In this section, we verify that the integration by parts formula we obtained for the directional derivatives is consistent with and generalizes the formulas known in the Riemannian case. Let us assume here that the foliation on $\M$ comes from a totally geodesic submersion $\pi : (\M,g) \to (\B,j)$ as in Example \ref{ex.RiemSubmersion}. Since the submersion has totally geodesic fibers, $\pi$ is harmonic and the projected process
\[
W_s^\B =\pi (W_s)
\]
is a Brownian motion on $\B$ started at $\pi(x)$. Observe that from the definition of submersion, the derivative map $T_x \pi$ is an isometry from $\mathcal{H}_x$ to $T_x \B$. From Example \ref{connection submersion},  the connection $\hat{\nabla}^\ve$ projects down to the Levi-Civita connection on $\B$. Therefore the stochastic parallel transport $\widehat{\Theta}_t^\varepsilon$ projects down to the stochastic parallel transport for the Levi-Civita connection along the paths of $\left\{ W_s^\B \right\}_{0 \leqslant s \leqslant 1}$. More precisely,

\[
\para_{0, s}=T_{W_s} \pi \circ \widehat{\Theta}_s^\varepsilon \circ (T_{x} \pi)^{-1},
\]
where $\para_{0,s} : T_{\pi (x)} \B \to T_{W_s^\B} \B$ is the stochastic parallel transport for the Levi-Civita connection along the paths of $\left\{ W_s^\B \right\}_{0 \leqslant s \leqslant 1}$. Consider now a Cameron-Martin process $\left\{ h(s)\right\}_{0\leqslant s \leqslant 1}$ in $T_{\pi (x)} \B$ and a cylinder function $F=f(W_{s_1}^\B,\cdots, W_{s_n}^\B)$ on $\B$. The function $F =f(\pi(W_{s_1}), \cdots, \pi(W_{s_n}))$ is then in $\mathcal{F}C^{\infty} \left( W_\mathcal{H} (\mathbb{M})\right)$(Refer to \cite[Definition 7.4]{Driver2004a}). Using Theorem \ref{IBP2}, one gets
\begin{align*}
  \mathbb{E}_x\left( \mathbf{D}_v F \ \right)
 =\mathbb{E}_x \left( F \int_0^1 \left\langle v_\mathcal{H}'(s)+ \frac{1}{2} (\widehat{\Theta}_s^{\varepsilon})^{-1} \mathfrak{Ric}_{\mathcal{H}}   \widehat{\Theta}_s^{\varepsilon} v_{\mathcal{H}} (s), dB_s\right\rangle_{\mathcal{H}}  \right),
\end{align*}
where $v_\mathcal{H}$ is the horizontal lift of $h$, that is, $v_\mathcal{H}=(T_{x} \pi)^{-1} h$. By definition, we have
\begin{align*}
\mathbf{D}_v F&=\sum_{i=1}^n \left\langle d_if(W^\B_{s_1},\cdots, W^\B_{s_n}),  (T_{W_{s_i}} \pi) \circ \widehat{\Theta}_{s_i}^{\varepsilon} v(s_i)  \right\rangle \\
  &=\sum_{i=1}^n \left\langle d_if(W^\B_{s_1}, \cdots, W^\B_{s_n}), \para_{0, s_i} h(s_i)  \right\rangle
\end{align*}
It is easy to check that $\mathfrak{Ric}_{\mathcal{H}}$ is the horizontal lift of the Ricci curvature $\mathfrak{Ric}^\B$ of $\B$. Therefore, the integration by parts formula for the directional derivative $\mathbf{D}_v F$ can be rewritten as follows.
\begin{align*}
  &  \mathbb{E}_{x} \left(\sum_{i=1}^n \left\langle d_if(W^\B_{s_1},\cdots, W^\B_{s_n}), \para_{0, s_i} h(s_i)  \right\rangle \right) \\
   =& \mathbb{E}_x \left( F \int_0^1 \left\langle h'(s)+ \frac{1}{2} \para_{0,s}^{-1} \mathfrak{Ric}^{\B}   \para_{0,s}  h(s),dB^\B_s\right\rangle_{T_{\pi(x)} \B}  \right),
\end{align*}
where $B^\B$ is the Brownian motion on $T_{\pi(x)} \B$ given by $B^\B =T_x \pi (B)$. This is exactly Driver's integration by parts formula in \cite{Driver1992b} for the Riemannian Brownian motion $X^\B$.

\subsubsection{K-contact manifolds}

In this section, we assume that the Riemannian foliation on $\M$ is the Reeb foliation of a K-contact structure. The Reeb vector field on $\M$ will be denoted by $R$ and the almost complex structure by $\mathbf{J}$. The torsion of the Bott connection is then
\[
T(X,Y)=\langle \mathbf J X, Y \rangle_\Ho R.
\]
Therefore with the previous notation, one has
\[
J_Z X =\langle Z,R \rangle \mathbf J X.
\]
and the vertical part of a tangent process is given by
 \begin{align*}
 v_\mathcal{V} (s)&=-\int_0^s (\widehat{\Theta}_r^{\varepsilon})^{-1} T ( \widehat{\Theta}_r^{\varepsilon} \circ dB_r, \widehat{\Theta}_r^{\varepsilon} v_\mathcal{H}(r)) \\
  &=\int_0^s ((\widehat{\Theta}_r^{\varepsilon})^{-1} R) \langle \mathbf J \widehat{\Theta}_r^{\varepsilon} v_\mathcal{H}(r), \widehat{\Theta}_r^{\varepsilon} \circ dB_r \rangle_\Ho.
 \end{align*}


\bibliographystyle{plain}	

\end{document}